\documentclass[12pt]{article}
\usepackage{tikz}
\usepackage{pgfplots}
\usepackage{times}
\usepackage{amsmath,amsfonts, amstext,amssymb,amsbsy,amsopn,amsthm}
\usepackage{dsfont}
\usepackage{esint}
\usepackage{graphicx}   
\usepackage[all]{xy}
\usepackage{color}
\usepackage{tikz}
\usepackage{lineno}
\usetikzlibrary{intersections}
\usepackage{hyperref}
\hypersetup{colorlinks=true,linkcolor=black,anchorcolor=black,citecolor=black}
\usepackage{titlesec}
\usepackage{caption}

\newtheorem{mthm}{Theorem}
\newtheorem{mpro}{Proposition}
\newtheorem{theorem}{Theorem}[section]
\newtheorem{mrem}{Remark}
\newtheorem{definition}{Definition}[section]

\newtheorem{mlem}{Lemma}
\newtheorem{lemma}[definition]{Lemma}

\newtheorem{mcor}{Corollary}

\theoremstyle{remark}
\newtheorem{remark}[definition]{Remark}
\numberwithin{equation}{section}

\newcommand{\R}{\mathbb{R}}

\newcommand{\fr}{\displaystyle\frac}
\newcommand{\jf}{\displaystyle\int}

\newcommand{\be}{\begin{equation}}
	\newcommand{\ee}{\end{equation}}
\newcommand{\bee}{\begin{equation*}}
	\newcommand{\eee}{\end{equation*}}
\newcommand{\N}{\mathbb{N}}
\newcommand{\Z}{\mathbb{Z}}
\DeclareMathOperator{\esssup}{esssup}

\title{Pointwise regularity of solutions for fully fractional parabolic equations}

\usepackage[T1]{fontenc}
\author{Yahong Guo, Qizhen Shen and Jiongduo Xie }
\date{ }
\usepackage{tikz}
\usetikzlibrary{positioning}
\usetikzlibrary{patterns}
\usepackage{pgfplots}
\pgfplotsset{compat=1.7}
\begin{document}
	\newcommand\tbbint{{-\mkern -16mu\int}}
	\newcommand\tbint{{\mathchar '26\mkern -14mu\int}}
	\newcommand\dbbint{{-\mkern -19mu\int}}
	\newcommand\dbint{{\mathchar '26\mkern -18mu\int}}
	\newcommand\bint{
		{\mathchoice{\dbint}{\tbint}{\tbint}{\tbint}}
	}
	\newcommand\bbint{
		{\mathchoice{\dbbint}{\tbbint}{\tbbint}{\tbbint}}
	}
    \maketitle
	
	\begin{abstract}
    This paper investigates  the higher pointwise regularity of nonnegative classical solutions for fully fractional parabolic equations $(\partial_t -\Delta)^{s} u =  f,$ where $s\in(0,1)$. We establish $C^{k+\alpha+2s}$ or $C^{k+\alpha+2s,\ln} (k\geq 0,\alpha\in[0,1))$ pointwise regularity according to $\alpha+2s\notin \Z$ or $\alpha+2s\in \Z$, which imply the classical local regularity directly. 
    We provide a simplified and unified proof by introducing novel equivalent definitions for pointwise function spaces. 
    Moreover, the equivalent integral representation and directional average for fractional heat kernel play an important role in our discussion.
    \end{abstract}
    \bigskip
    
    \bigskip
    
    \section{Introduction}
    
    In this paper, we establish new pointwise regularity estimates for the solutions to the following master equations   
\begin{equation}\label{1.0}
    	(\partial_t -\Delta)^{s} u(x,t) =  f(x,t)\ \ \mbox{in}\ \ \R^n\times\R.
    \end{equation}
Here, the fully fractional heat operator $(\partial_t-\Delta)^s$ was first introduced by M. Riesz in \cite{Riesz} and is defined by the following singular integral
\begin{equation}\label{nonlocaloper}
(\partial_t - \Delta)^s u(x,t)
:= C_{n,s} \int_{-\infty}^{t} \int_{\mathbb{R}^n}
\frac{u(x,t) - u(y,\tau)}{(t-\tau)^{\frac{n}{2} + 1 + s}} e^{-\frac{|x-y|^2}{4(t-\tau)}}  dy  d\tau,
\end{equation}
    where $0<s<1$, the integral with respect to $y$ is in the sense of Cauchy principal value, and
    the normalization constant is given by $$C_{n,s}=\frac{1}{(4\pi)^{\frac{n}{2}}|\Gamma(-s)|}$$
    with $\Gamma(\cdot)$ denoting the Gamma function. We say $u(x,t)$ is a classical solution of (\ref{1.0}) provided 
    $$u(x,t)\in C^{2s+\epsilon,s+\epsilon}_{x,\, t,\, {\rm loc}}(\mathbb{R}^n\times\mathbb{R}) \cap \mathcal{L}(\mathbb{R}^n\times\mathbb{R})$$
    for some $\varepsilon >0$, where the slowly increasing function space $\mathcal{L}(\R^n\times\mathbb{R})$ is defined by
    $$\mathcal{L}(\R^n\times\mathbb{R})=\left\{u(x,t)\in L_{\mathrm{loc}}^1(\mathbb{R}^n\times\mathbb{R})\mid\int_{-\infty}^t\int_{\mathbb{R}^n}\frac{|u(x,\tau)|e^{-\frac{|x|^2}{4(t-\tau)}}}{1+(t-\tau)^{\frac{n}{2}+1+s}}dxd\tau<\infty,\:\forall\:t\in\mathbb{R}\right\},$$
   and the local parabolic H\"{o}lder space $C^{2s+\epsilon,s+\epsilon}_{x,\, t,\, {\rm loc}}(\mathbb{R}^n\times\mathbb{R})$ is standard as given in \cite{WGC}.   
   
   The fractional operator $(\partial_t-\Delta)^s$ demonstrates essential non-local characteristics in both spatial and temporal dimensions, and converges to the classical heat operator $\partial_t-\Delta$ as $s\rightarrow 1$ (cf. \cite{FNW}). Furthermore, when applied to a function $u$ depending solely on the spatial variable $x$, the space-time nonlocal operator $(\partial_t-\Delta)^s$ reduces to the well-known fractional Laplacian
\begin{equation}\label{1.3-1}
    	(\partial_t-\Delta)^s u(x)=(-\Delta)^s u(x):=C_{n, s}P.V. \int_{\mathbb{R}^n}\frac{u(x)-u(y)}{|x-y|^{n+2s}}dy.
    \end{equation}
This operator has attracted significant attention, and we refer the reader to \cite{CS,CL, CLL1, CZhu, DQ, DDW, DFV, LXYZ, SV} and the references therein for a comprehensive overview of recent developments. On the other hand, when $(\partial_t-\Delta)^s$ is applied to a function $u$ that depends only on the time variable $t$, it simplifies to the Marchaud fractional derivative $\partial_t^s$:
    \begin{equation}\label{1.4-1}
    	(\partial_t-\Delta)^s u(t)=\partial_t^s u(t):=C_s \jf_{-\infty}^t\fr{u(t)-u(\tau)}{(t-\tau)^{1+s}}d\tau.
    \end{equation}
This differential operator emerges naturally in the mathematical modeling of diverse physical systems (see \cite{DCL1, DCL2, EE}).

Equation \eqref{1.0} plays a fundamental role in numerous applications in the physical and biological sciences.  These include anomalous diffusion \cite{KBS,MK}, chaotic dynamics \cite{Z} and biological invasions \cite{BRR}. In recent years, there has been extensive research on both the qualitative and quantitative aspects of \eqref{1.0} (see \cite{CS-2014,CG,CG2,CG1,WGC,ST}).  

Schauder theory plays a central role in the study of nonlocal elliptic and parabolic equations, providing essential regularity estimates that are critical for establishing well-posedness and analyzing large-time behavior.
The relevant work involving nonlocal space-time operators \eqref{1.3-1} and \eqref{1.4-1} has been widely explored, and we refer to 
\cite{A-2018,ACV,SHJ-2023,CCV,CS-2009,CLM-2020, DJ,DZ,FR-2024,JX-2015,S-2007,S-2024,R-2013}.

For equation (\ref{1.0}), classical Schauder estimates 
\begin{equation} \label{ST2}\|u\|_{C_{x,t}^{2s + \alpha, s + \alpha/2}(Q_1(0,0))} \leq C \left( \|f\|_{C_{x,t}^{\alpha, \alpha/2}(Q_2(0,0))} + \|u\|_{L^\infty (\bf{\R^n \times (-\infty, 4)})}\right) .\end{equation}
is established in pioneering work by Caffarelli and Silvestre \cite{CS-2014},
Stinga and Torrea \cite{ST}, 
and subsequent works \cite{B-2020,BS,KO}.
Recently,  for nonnegative solutions, Chen, Guo and Li \cite{WGC} established the refined Schauder estimates that depend solely on local data, namely, by replacing $\|u\|_{L^\infty (\bf{\R^n \times (-\infty, 4)})}$ with $\|u\|_{L^\infty (Q_{2}(0,0))}$ in \eqref{ST2}.
A powerful technique for deriving classical Schauder estimates is approximating by polynomials, which naturally leads to a general concept "pointwise Schauder estimate".
Recently, the pointwise Schauder estimate has attracted growing attention \cite{A-1997,C-1989,J,JM,JX-2016,LY-2024,LL-2020,O-2013} 
and found significant applications in other critical areas, including free boundary problems \cite{LK-2023,EM-2013,LM-2009,MZ}, nodal set \cite{H-2000,Lianyy} and transmision problems \cite{SS}.

The pointwise regularity of viscosity solutions to the Laplace equation was first established by Caffarelli \cite{C-1989}. This result has recently been improved by Lian and Zhang \cite{LK-2023}, who proved a $C^{k+l+\alpha}$ estimate for Dirichlet problems, where $k,l \in \mathbb{N}$ and $0<\alpha<1$.

For the classical case corresponding to $s=1$ in equation (\ref{1.0}), at a given point $(x_{0},t_{0})$, when $f(x,t)\in C^{k+\alpha}(x_{0},t_{0})$ with $0<\alpha<1$, Lian \cite{Lianyy} established  $C^{k+\alpha+2}(x_{0},t_{0})$ pointwise estimate for weak solutions
by  employing compactness and perturbation techniques. 


Furthermore, for the nonlocal stationary case \eqref{1.3-1}, Li and Wu \cite{LW1}, by utilizing the Poisson representation, proved $C^{k+\alpha+2s}(x_{0})$ pointwise regularity for weak solution $u$ to
\begin{equation}\label{Wu}
(-\Delta)^{s}u(x)=f(x)
\end{equation}
under assumption $f\in C^{k+\alpha}(x_{0})$
at a given point $x_{0}$. 

When handling the general equation (\ref{1.0}), we will face two significant difficulties:

\begin{itemize}
    \item [(i)] One of the key ingredients in \cite{Lianyy} is that when the local heat operator ${\partial_t}-\Delta$ applies to a quadratic polynomial, it yields a constant. However, this is no longer true for our nonlocal operator $(\partial_t-\Delta)^s$ here. \item[(ii)] In \cite{LW1}, an important tool is the Poisson representation of $s-$harmonic functions. While for our fully fractional parabolic operator $(\partial_t-\Delta)^s$, there is no corresponding Poisson kernel. 
\end{itemize}
 To overcome these difficulties, some new ideas and techniques are introduced. We employ the integral representation of the solution and decompose it into two parts: an internal part and an external part. For the external part, different from \cite{Lianyy}, we provide a new perturbation technique for the kernel function to control the value of the solution at one point  by its nearby points. For the internal part, we decompose the internal part into three components $S_{r},T_{r}$ and $u_{P}$ (See Definition \ref{DEF}). The argument for $T_{r}$
  is simplified and unified by introducing some novel equivalent definitions of pointwise function spaces.

\bigskip

Before stating our main results, we introduce some relevant pointwise function spaces.
\subsection{Notation and definitions }

 Let
    $$\sigma=(\sigma_{1},\sigma_{2},\cdots,\sigma_{n+1})$$ be a multi-index with $\sigma_{i}\in \N~(i=1,2,\cdots,n+1)$.
    For $(x,t)\in\R^{n}\times\R$ and $k\in \N$, we denote
    \begin{align*}
    	&|\sigma^{\prime}|=\sum_{i=1}^{n}\sigma_{i},\quad|\sigma|=|\sigma^{\prime}|+2\sigma_{n+1},
    	\quad\sigma!=\prod_{i=1}^{n+1}(\sigma_{i}!),\quad(x,t)^{\sigma}=\prod_{i=1}^{n}x_{i}^
    	{\sigma_{i}}\cdot t^{\sigma_{n+1}}, \\
    	&D^{\sigma}u=\frac{\partial^{|\sigma^{\prime}|+\sigma_{n+1}}u}{\partial x_{1}^{\sigma_{1}}\cdots\partial x_{n}^{\sigma_{n}}\partial t^{\sigma_{n+1}}},\quad|D^{k}u|=\sum_{|\sigma|=k}|
    	D^{\sigma}u|.
    \end{align*}
\begin{definition}[Parabolic polynomial\ \cite{Lianyy}]\label{poly}

    Let $\mathcal{P}_{k}(k\geq0)$ denote the space of parabolic polynomials of degree at most $k$, defined as all functions of the form: 
    \begin{align}
    	P(x,t)=\sum_{|\sigma|\leq k}\frac{a_{\sigma}}{\sigma!}(x,t)^{\sigma},
    \end{align}
    where $\sigma$ is a multi-index and
    $a_{\sigma} \in \mathbb{R}$ are constant coefficients.

    The associated norm at $(x_{0},t_{0})$ is then defined by:
    \begin{align}    	\|P\|_{(x_{0},t_{0})}=\sum_{j=0}^{k}|D^{j}P(x_{0},t_{0})|.
    \end{align}
\end{definition}
Let 
$$Q_{r}(x_{0},t_{0})=B_{r}(x_{0})\times(t_{0}-r^{2},t_{0}]$$ 
and 
$$\widetilde{Q}_{r}(x_{0},t_{0})=B_{r}(x_{0})\times(t_{0}-r^{2},t_{0}+r^{2}).$$ Now a function of $C^{k,\alpha}$ at a given point $(x_{0},t_{0})$ is defined as follows.
\begin{definition}\label{sec:space-a}

    Let $U\subset\mathbb{R}^{n+1}$ be a bounded domain and $f:U\to\mathbb{R}$ be a function. Given parameters $1\leq s_{1},s_{2}<\infty, 0\leq\alpha<1, k\in \N$ and radius $r_{0}>0$, let $Q_{r_{0}}(x_{0},t_{0})\subset U$ and
    \begin{equation}\label{nu-f}
        \nu_{f}(R)=\sup_{0<r\leq R}\left(\bbint_{t_{0}-r^{2}}^{t_{0}}\left(\bbint_{ B_r(x_{0})}|f(x,t)-P(x,t)|^{s_{1}}dx\right)^{\frac{s_{2}}{s_{1}}}dt\right)^{\frac{1}{s_{2}}},
    \end{equation}
    where $P$ is a given polynomial.

    We say that $f\in C_{s_{1},s_{2}}^{k,\alpha}(x_{0},t_{0};r_{0})$, if there exist a  constant $K=K_{x_{0},t_{0}}>0$ and a polynomial $P=P_{x_{0},t_{0}}\in\mathcal{P}_{k}$ such that for any $r\in[1/2,1/4]$ and each integer $m\geq 1$, it holds
    \begin{align}\label{1.1000}
    		\sum_{i=0}^{m-1}(r_{0}r)^{-i(k+\alpha)}\nu_{f}\left(r_{0}r^{i}\right)\leq Km.
    \end{align}
    The associated norm is defined as:
    \begin{align}
    	\|f\|_{C_{s_{1},s_{2}}^{k,\alpha}(x_{0},t_{0};r_{0})}=\min_{_{K\geq0;P\in\mathcal{P}_{k}}} (K+\|P\|_{({x_{0},t_{0}})}).
    \end{align}

For some open set $U^{\prime}\subset\subset U$, we say that  $f\in C_{s_{1},s_{2}}^{k,\alpha}(\overline{U^{\prime}}),$ if there exists $r_0\in (0,1]$ such that for all $(x,t)\in U^{\prime},$ we have $f\in C_{s_{1},s_{2}}^{k,\alpha}(x,t;r_{0})$, and  
the associated norm
    \begin{align}
    	\|f\|_{C_{s_{1},s_{2}}^{k,\alpha}(\overline{U'})}:=\sup\limits_{(x,t)\in U'}\|f\|_{C_{s_{1},s_{2}}^{k,\alpha}(x,t;r_{0})}<\infty.
    \end{align}
        In addition, if $s_{1}=\infty$, we replace the average integral with respect to $x$ in (\ref{1.1000}) into $\mathop{\esssup}_{B_{r}(x_{0})}$. Similarly, if $s_{2}=\infty$, we replace the average integral with respect to $t$ of (\ref{1.1000}) into $\mathop{\esssup}_{(t_{0}-r^{2},t_{0}]}$.
        We denote $C_{s_{1},s_{2}}^{\infty}(x_{0},t_{0};r_{0})=\cap_{k\geq 0}C_{s_{1},s_{2}}^{k}(x_{0},t_{0};r_{0})$.
    \end{definition}
\begin{mrem}

   If  (\ref{1.1000}) is replaced by 
    \begin{eqnarray}
    	\begin{aligned}
    		\sum_{i=0}^{\infty}(r_{0}r)^{-i(k+\alpha)}\nu_{f}\left(r_{0}r^{i}\right)\leq K,
    	\end{aligned}
    \end{eqnarray}
    we 
    call $f\in C_{s_{1},s_{2}}^{k,\alpha,Dini}(x_{0},t_{0};r_{0})$. 
    Similarly, we say $f \in C_{s_{1},s_{2}}^{k,\alpha,\ln}(x_{0},t_{0};r_{0}),$ if  we replace (\ref{1.1000}) with 
    \begin{eqnarray}
    	\begin{aligned}
    		\sum_{i=0}^{m-1}(r_{0}r)^{-i(k+\alpha)}\nu_{f}\left(r_{0}r^{i}\right)\leq Km^{2}~\text{for any }  m\in\N^*.
    	\end{aligned}
    \end{eqnarray}
\end{mrem}
\begin{mrem}\label{remark 2}
The above definitions of $C_{s_{1},s_{2}}^{k,\alpha}(x_{0},t_{0};r_{0})$, $C_{s_{1},s_{2}}^{k,\alpha,Dini}(x_{0},t_{0};r_{0})$ and $C_{s_{1},s_{2}}^{k,\alpha,\ln}(x_{0},t_{0};r_{0})$ are equivalent to the usual ones; we refer to the appendix for details. Moreover, for all $1\leq s_{1},s_{2}\leq\infty$, the spaces $C_{s_{1},s_{2}}^{k,\alpha}(\overline{U'})$ coincide with $C^{k,\alpha}(\overline{U'})$ as shown in \cite{Lianyy}.  
\end{mrem}
Similar to $C^{\log L,s+\frac{\alpha}{2}}$ defined in \cite{WGC}, we introduce a new pointwise function space:
\begin{definition}
    Let $U\subset\mathbb{R}^{n+1}$ be a bounded domain and $f:U\to\mathbb{R}$ be a function. Given $(x_{0},t_{0})\in U, r_{0}>0, k\in \N$, we say that $f\in C^{k,x-\ln}(x_{0},t_{0};r_{0})$, if there exist a constant $K=K_{x_{0},t_{0}}>0$ and a polynomial $P(x,t)=P_{x_{0},t_{0}}(x,t)\in\mathcal{P}_{k}$ such that for all $(x,t)\in Q_{r_{0}}(x_{0},t_{0})\subset U$ and $0<r\leq r_{0}$, it holds that
    \begin{eqnarray}
    	\begin{aligned}\label{1.12}
    		|f(x,t)-P(x,t)|\leq Kr^{k}\max\left\{1,|\ln
        |x-x_{0}\|\right\},
    	\end{aligned}
    \end{eqnarray}
    where $r=(|x-x_{0}|^{2}+|t-t_{0}|)^{\frac{1}{2}}$.

    The associated norm is defined as:
    \begin{align}
    	\|f\|_{C^{k,x-\ln}(x_{0},t_{0};r_{0})}=\min_{_{K\geq0;P\in\mathcal{P}_{k}}
        } (K+\|P\|_{x_{0},t_{0}}).
    \end{align}

    Given some open set $U^{\prime}\subset\subset U$, we say $f\in C^{k,x-\ln}(\overline{U^{\prime}})$, if there exists a $r_0\in (0,1]$, for all $(x,t)\in U'$, 
    $f\in C^{k,x-\ln}(x,t;r_{0})$, and the associated norm
    \begin{align}
    	\|f\|_{C^{k,x-\ln}(\overline{U'})}:=\sup\limits_{(x,t)\in U'}\|f\|_{C^{k,x-\ln}(x,t;r_{0})}<\infty.
    \end{align}
\end{definition}
Similarly, we can define time-independent function spaces $C^{k,\alpha}_{s_{1}}(x_{0};r_{0})$, $C^{k,\alpha,\ln}_{s_{1}}(x_{0};r_{0})$, $C^{k,\alpha,Dini}_{s_{1}}(x_{0};r_{0})$ and space-independent ones $C^{k,\alpha}_{s_{2}}(t_{0};r_{0})$, $C^{k,\alpha,\ln}_{s_{2}}(t_{0};r_{0})$, $C^{k,\alpha,Dini}_{s_{2}}(t_{0};r_{0})$.
\subsection{Main Results}
    Now we are ready to state our main results on the pointwise regularity for equation \eqref{1.0}. For simplicity, we denote $C_{s_{1},s_{1}}^{k,\alpha}(x_{0},t_{0};r_{0})$ by 
$C_{s_{1}}^{k,\alpha}(x_{0},t_{0};r_{0})$, $C_{\infty}^{k,\alpha}(x_{0},t_{0};r_{0})$ by $C^{k,\alpha}(x_{0},t_{0};r_{0})$ and $C_{s_{1},s_{2}}^{k,\alpha}(x_{0},t_{0};r_{0})$ by $C_{s_{1},s_{2}}^{k+\alpha}(x_{0},t_{0};r_{0})$, similar simplifications work for other pointwise function  spaces.

\begin{mthm}\label{th6}
    	Assume that $f$ is a nonnegative function in $\R^{n}\times\R$ and  $u$ is a nonnegative solution  of (\ref{1.0}). Then there exists a constant $C=C(n,k,\alpha,s)>0$ such that for each $(x_0,t_0) \in \mathbb{R}^n \times \mathbb{R}$, if $f\in C_{1}^{k+\alpha}(x_{0},t_{0};1)$, then
    	\begin{enumerate}
    	\item[(a)] for $2s+\alpha\notin\Z$, $u\in C_{1}^{k+\alpha+2s}(x_{0},t_{0};\frac{1}{2})$, and 
    	\begin{equation}\label{th6-u}
    		\|u\|_{C_{1}^{k+\alpha+2s}(x_{0},t_{0};\frac{1}{2})}\leq C(\|f\|_{C_{1}^{k+\alpha}(x_{0},t_{0};1)}+\|u\|_{L^{\infty}(\widetilde{Q}_{\frac{1+\sqrt{n}}{2\sqrt{n}}}(x_{0},t_{0}))});
    	\end{equation}
    	\item[(b)] for $2s+\alpha\in\Z$,  
        
        when $k+2s+\alpha$ is odd,  $u\in C_{1}^{k+\alpha+2s,x-\ln}(x_{0},t_{0};\frac{1}{2})$, and
    	\begin{align}
    		\|u\|_{C_{1}^{k+\alpha+2s,x-\ln}(x_{0},t_{0};\frac{1}{2})}\leq C(\|f\|_{C_{1}^{k+\alpha}(x_{0},t_{0};1)}+\|u\|_{L^{\infty}(\widetilde{Q}_{\frac{1+\sqrt{n}}{2\sqrt{n}}}(x_{0},t_{0}))});
    	\end{align}

        when $k+2s+\alpha$ is even, $u\in C_{1}^{k+\alpha+2s,\ln}(x_{0},t_{0};\frac{1}{2})$, and
    	\begin{align}
    		\|u\|_{C_{1}^{k+\alpha+2s,\ln}(x_{0},t_{0};\frac{1}{2})}\leq C(\|f\|_{C_{1}^{k+\alpha}(x_{0},t_{0};1)}+\|u\|_{L^{\infty}(\widetilde{Q}_{\frac{1+\sqrt{n}}{2\sqrt{n}}}(x_{0},t_{0}))}).
    	\end{align}
        \item[~]Furthermore, under the stronger condition $f\in C_{1}^{k+\alpha,Dini}(x_{0},t_{0};1)$, 
        there holds  $u\in C_{1}^{k+\alpha+2s}(x_{0},t_{0};\frac{1}{2})$, and
    	\begin{align}
    		\|u\|_{C_{1}^{k+\alpha+2s}(x_{0},t_{0};\frac{1}{2})}\leq C(\|f\|_{C_{1}^{k+\alpha,Dini}(x_{0},t_{0};1)}+\|u\|_{L^{\infty}(\widetilde{Q}_{\frac{1+\sqrt{n}}{2\sqrt{n}}}(x_{0},t_{0}))}).
    	\end{align}
        \end{enumerate}
    \end{mthm}
Based on the proof of the above theorem, 
by rescaling, we can derive the following more general result. Here we only state the case $f\in C_{s_{1},s_{2}}^{k+\alpha,Dini}(x_{0},t_{0};r_0)$, similar conclusions hold for other cases when $f\in C_{s_1,s_2}^{k+\alpha}(x_{0},t_{0};r_0).$

    \begin{mcor}\label{cor1}
    	For some $k\in\N,\alpha\in[0,1)$ and $(x_0,t_0) \in \mathbb{R}^n \times \mathbb{R}$, assume that $f$ is a nonnegative function in $\R^{n}\times\R$ with $f\in C_{s_{1},s_{2}}^{k+\alpha,Dini}(x_{0},t_{0};r_0)$ and  $u$ is a nonnegative solution of (\ref{1.0}). Then there exists a positive constant $C=C(n,k,\alpha,s,s_{1},s_{2})$ such that
    	\begin{align}
    		\|u\|_{C^{k+\alpha+2s}_{s_{1},s_{2}}(x_{0},t_{0};\frac{r_{0}}{2})}\leq C(\|f\|_{C_{s_{1},s_{2}}^{k+\alpha,Dini}(x_{0},t_{0};r_{0})}+r_{0}^{-(k+\alpha+2s)}\|u\|_{L^{\infty}(\widetilde{Q}_{\frac{1+\sqrt{n}}{2\sqrt{n}}r_{0}}(x_{0},t_{0}))}).
    	\end{align}
     \end{mcor}
The pseudodifferential equation (\ref{1.0}) is equivalent to the integral equation 
\begin{equation} \label{inteq}
    	u(x,t) = c + \int_{-\infty}^t \int_{\mathbb{R}^n} f(y,\tau) K_{-s}(x-y, t-\tau) dy d\tau,
    \end{equation}
in the case when $f(x,t)$ and $u(x,t)$ is nonnegative (see \cite{WGC}), where 
\[K_{-s}(x-y,t-\tau):=C_{n,s}
\frac{e^{-\frac{|x-y|^2}{4(t-\tau)}}}{(t-\tau)^{\frac{n}{2} + 1 - s}}.   \]
This integral equation is an essential tool that we will employ to derive a series of pointwise estimates for the solutions of (\ref{1.0}).
    \begin{mrem}
    If we directly consider the solution of integral equation (\ref{inteq}), we can obtain similar results by requiring $f(x,t)\in C^{k,\alpha}(x_{0},t_{0})$ with fast decay at infinity such as
    \begin{align}
    \int_{-\infty}^{\infty} \int_{\mathbb{R}^n} \frac{|f(x,t)|}{|t|^{\frac{n}{2}+1-s}+|x|^{n+2-2s}}\operatorname{d}\!x\operatorname{d}\!t<\infty.
    \end{align}
    \end{mrem}
     
    In particular, for the stationary case \eqref{Wu}, we derive stronger pointwise regularity results than those obtained in \cite{LW1} by imposing the additional assumptions $u, f\geq 0$. Note that in this case, $\mathcal{L}(\R^n\times\mathbb{R})$ is reduced to $\mathcal{L}^{2s}$ defined in \cite{LW1}.
    \begin{mcor}\label{x}
    	For some $k\in\N,\alpha\in[0,1)$ and $x_{0}\in\R^{n}$, suppose $f$ is a nonnegative function in $\R^{n}$ with $f\in C_{1}^{k+\alpha}(x_{0};1)$ and $u\in \mathcal{L}^{2s}$ is a nonnegative solution of \eqref{Wu} in $\R^{n}$. 
        Then there exists a positive constant $C=C(n,k,\alpha,s)$ such that
        \begin{enumerate}
            \item[(a)] for $2s+\alpha\notin \Z,$ $u\in C_{1}^{k+\alpha+2s}(x_{0};\frac{1}{2}),$ and
    	\begin{align}
    		\|u\|_{C_{1}^{k+\alpha+2s}(x_{0};\frac{1}{2})}\leq C(\|f\|_{C_{1}^{k+\alpha}(x_{0};1)}+\|u\|_{L^{\infty}(B_{1}(x_{0}))});
    	\end{align}
    	\item[(b)] for $k+2s+\alpha\in \Z$, $u\in C_{1}^{k+\alpha+2s,\ln}(x_{0};\frac{1}{2}),$
    	and
    	\begin{align}
    		\|u\|_{C_{1}^{k+\alpha+2s,\ln}(x_{0};\frac{1}{2})}\leq C(\|f\|_{C_{1}^{k+\alpha}(x_{0};1)}+\|u\|_{L^{\infty}(B_{1}(x_{0}))}).
    	\end{align}
    	\item[~] In addition, under the stronger assumption $f\in C_{1}^{k+\alpha,Dini}(x_{0};1)$, it holds that $u\in C_{1}^{k+\alpha+2s}(x_{0};\frac{1}{2}),$
    	and
    	\begin{align}
    		\|u\|_{C_{1}^{k+\alpha+2s}(x_{0};\frac{1}{2})}\leq C(\|f\|_{C_{1}^{k+\alpha,Dini}(x_{0};1)}+\|u\|_{L^{\infty}(B_{1}(x_{0}))}).
    	\end{align}
        \end{enumerate}
    \end{mcor}
    While for the space‑independent case, we obtain a novel analogous result. 
    \begin{mcor}\label{t}
      For some $k\in\N,\alpha\in[0,1)$ and $t_{0}\in\R$, assume that $f$ is a nonnegative function in $\R$ with $f\in C_{1}^{k+\alpha}(t_{0};1)$ and that $u$ is a nonnegative solution of $$\partial_{t}^{s}u(t)=f(t), ~~t\in \R.$$
      Then there exists a positive constant $C=C(n,k,\alpha,s)$ such that
        \begin{enumerate}
            \item[(a)] for $s+\alpha\notin \Z$, $u\in C_{1}^{k+\alpha+s}(t_{0};\frac{1}{2}),$
    	and
    	\begin{align}
    		\|u\|_{C_{1}^{k+\alpha+s}(t_{0};\frac{1}{2})}\leq C(\|f\|_{C_{1}^{k+\alpha}(t_{0};1)}+\|u\|_{L^{\infty}((t_{0}-1,t_{0}+1))});
    	\end{align}
    	\item[(b)] for $s+\alpha\in \Z$, $u\in C_{1}^{k+\alpha+s,\ln}(t_{0};\frac{1}{2}),$
    	and
    	\begin{align}
    		\|u\|_{C_{1}^{k+\alpha+s,\ln}(t_{0};\frac{1}{2})}\leq C(\|f\|_{C_{1}^{k+\alpha}(t_{0};1)}+\|u\|_{L^{\infty}((t_{0}-1,t_{0}+1))}).
    	\end{align}
    	\item[~]Moreover, under the stronger condition $f\in C_{1}^{k+\alpha,Dini}(t_{0};1)$, it holds that $u\in C_{1}^{k+\alpha+s}(t_{0};\frac{1}{2}),$ and
    	\begin{align}
    		\|u\|_{C_{1}^{k+\alpha+s}(t_{0};\frac{1}{2})}\leq C(\|f\|_{C_{1}^{k+\alpha;Dini}(t_{0};1)}+\|u\|_{L^{\infty}((t_{0}-1,t_{0}+1))}).
    	\end{align}
        \end{enumerate}
    \end{mcor}
    Next, by the equivalence between local and pointwise regularity explained in Remark \ref{remark 2} and Corollary \ref{cor1}, we further obtain the classical (local) Schauder estimates for equation \eqref{1.0}. This recovers the results of \cite{WGC} and, moreover, provides the corresponding estimates for the case $f\in C^{k+\alpha,Dini}(\widetilde{Q}_{1}(x_{0},t_{0}))$.
    \begin{mcor}[Classical Schauder estimate]
    	For some $k\in\N,\alpha\in[0,1)$ and $(x_{0},t_{0})\in\R^{n}\times\R$, assume that $f$ is a nonnegative function in $\R^{n}\times\R$ with $f\in C^{k+\alpha}(\widetilde{Q}_{1}(x_{0},t_{0})),$ and that $u$ is a nonnegative solution of \eqref{1.0}. Then there exists a positive constant $C=C(n,k,\alpha,s)$ such that
        \begin{enumerate}
            \item[(a)] for $2s+\alpha\notin \Z,u\in C^{k+\alpha+2s}(\widetilde{Q}_{1/2}(x_{0},t_{0})),$
    	and
    	\begin{align}
    		\|u\|_{C^{k+\alpha+2s}(\widetilde{Q}_{\frac{1}{2}}(x_{0},t_{0}))}\leq C(\|f\|_{C^{k+\alpha}(\widetilde{Q}_{1}(x_{0},t_{0}))}+\|u\|_{L^{\infty}(\widetilde{Q}_{1}(x_{0},t_{0}))});
    	\end{align}
    	\item[(b)] for $2s+\alpha\in \Z$, 
        
        when $k+2s+\alpha$ is odd, $u\in C^{k+\alpha+2s,x-\ln}(\widetilde{Q}_{1/2}(x_{0},t_{0})),$
    	and
    	\begin{align}
    		\|u\|_{C^{k+\alpha+2s,x-\ln}(\widetilde{Q}_{\frac{1}{2}}(x_{0},t_{0}))}\leq C(\|f\|_{C^{k+\alpha}(\widetilde{Q}_{1}(x_{0},t_{0}))}+\|u\|_{L^{\infty}(\widetilde{Q}_{1}(x_{0},t_{0}))});
    	\end{align}
        when $k+2s+\alpha$ is even, $u\in C^{k+\alpha+2s,\ln}(\widetilde{Q}_{1/2}(x_{0},t_{0})),$
    	and
    	\begin{align}
    		\|u\|_{C^{k+\alpha+2s,\ln}(\widetilde{Q}_{\frac{1}{2}}(x_{0},t_{0}))}\leq C(\|f\|_{C^{k+\alpha}(\widetilde{Q}_{1}(x_{0},t_{0}))}+\|u\|_{L^{\infty}(\widetilde{Q}_{1}(x_{0},t_{0}))}).
    	\end{align}
    	\item[~]Moreover, under the stronger condition $f\in C^{k+\alpha,Dini}(\widetilde{Q}_{1}(x_{0},t_{0}))$, there  holds $u\in C^{k+\alpha+2s}(\widetilde{Q}_{1/2}(x_{0},t_{0})),$
    	and
    	\begin{align}
    		\|u\|_{C^{k+\alpha+2s}(\widetilde{Q}_{\frac{1}{2}}(x_{0},t_{0}))}\leq C(\|f\|_{C^{k+\alpha,Dini}(\widetilde{Q}_{1}(x_{0},t_{0}))}+\|u\|_{L^{\infty}(\widetilde{Q}_{1}(x_{0},t_{0}))}).
    	\end{align}
        \end{enumerate}
    \end{mcor}

\subsection{Main Ideas of the Proof}
In this subsection, we sketch the proof of Theorem \ref{th6}.
Specifically, we work on the integral equation \eqref{inteq} and divide the integral into two parts. Without loss of generality, we may assume $c=0$ .  
For any measurable set $A\subset \R^{n+1}$, we define  $$f_{A}(x,t)=f(x,t)\chi_{A}(x,t),$$ where $\chi_{A}$ denotes the characteristic function of $A$. 

    For each $(x_0,t_0) \in \mathbb{R}^n \times \mathbb{R}$ and $0<r\leq 1,$ we decompose  a solution $u(x,t)$ of (\ref{inteq}) into two parts with respect to  $\widetilde{Q}_r(x_{0},t_{0})$ (See Figure \ref{figure 1}): the external function
    	\begin{align}\label{1.24}
    		v_{r}(x,t)=\int_{-\infty}^{t}\int_{\R^{n}}f_{\widetilde{Q}_{r}^{c}(x_{0},t_{0})}(y,\tau)K_{-s}(x-y,t-\tau)dyd\tau,
    	\end{align}
and the internal function
    	\begin{align}\label{1.25-1}
    		w_{r}(x,t)=\int_{-\infty}^{t}\int_{\R^{n}}f_{\widetilde{Q}_{r}(x_{0},t_{0})}(y,\tau)K_{-s}(x-y,t-\tau)dyd\tau.
    	\end{align}
We consider the case where $r=1$. Theorem \ref{th6} will easily follow from Proposition \ref{v-E} and Proposition \ref{w(x,t)}, which consider the pointwise regularity of $v_{1}$ and $w_{1}$, respectively. 
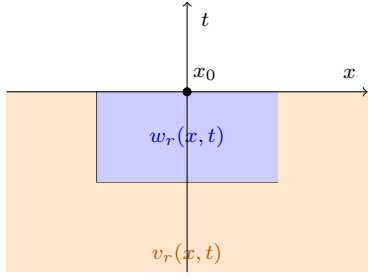
\begin{figure}[h]
\begin{center}
\begin{tikzpicture}[scale=1.2]
	\fill[white] (-2,-2) rectangle (2,1);
	\begin{scope}
		\clip (-2,-2) rectangle (2,1); 
		\fill[orange!20, draw=none] (-2,-2) rectangle (2,0);
		\fill[white] (0,0) circle (1);	
\end{scope}
	\draw [black] (-1,-1) rectangle (1,0);
	\fill[blue!20] (-1,-1) rectangle (1,0);
	\node[orange!70!black] at (0,-1.8) {\scriptsize $v_{r}(x,t)$}; 
	\node[blue!70!black] at (0,-0.5) {\scriptsize $w_{r}(x,t)$}; 
	\draw[black,thin,->] (-2,0)--(2,0);
    \draw[black,thin,->] (0,-2)--(0,1);
	\fill (0,0) circle(0.05);
	\node[black] at (1.8,0.2) {\scriptsize $x$};
    \node[black] at (0.2,0.8) {\scriptsize $t$};
	\node[black] at (0.2,0.2) {\scriptsize $x_0$};      
\end{tikzpicture}
\caption{Partition for the solution}
\label{figure 1}
\end{center}
\end{figure}
   
   We first derive the pointwise $C^{k}$ estimate of $v_{1}$.
    \begin{mpro}\label{v-E}
    	For some $x_{0},t_{0}\in\R^{n}\times\R$, let $u$ be a nonnegative solution of equation \eqref{1.0}
    	with $f\geq 0$ and let $v_{1}$ be defined by \eqref{1.24} for $r=1$. 
        Then $v_{1}\in C^{\infty}(x_{0},t_{0};1/2)$, and for any $k\in\mathbb{N}$, there exists a positive constant $C=C(n,k)$ such that
    		\begin{align}
    			\|v_{1}\|_{C^{k}(x_{0},t_{0};\frac{1}{2})}\leq C\|u\|_{L^{\infty}(\widetilde{Q}_{\frac{1+\sqrt{n}}{2\sqrt{n}}}(x_{0},t_{0}))}.	
    		\end{align}
    \end{mpro}
We will apply the perturbation technique to the kernel function to prove Proposition \ref{v-E}. For clarity of exposition, we consider a simple case with $n=2$ and $(x_{0},t_{0})=(0,0)$. 

Estimating the second derivative $v_{1,x_{1}x_{1}}(x,t)$ in terms of $v_{1}(x,t)$ itself presents a substantial difficulty. Indeed, by definition \eqref{1.24} with $r=1$, that is, 
\begin{align}\label{1.27}
    		v_{1}(x,t)=\int_{-\infty}^{t}\int_{\R^{n}}f_{\widetilde{Q}_{1}^{c}(0,0)}(y,\tau)K_{-s}(x-y,t-\tau)dyd\tau,
\end{align}
it follows that
\begin{align}\label{1.28}
    v_{1,x_1x_1}(x,t)=\int_{-\infty}^{t}\int_{\R^{n}}f_{\widetilde{Q}_{1}^{c}(0,0)}(y,\tau)\left[\frac{1}{2(t-\tau)} +\frac{(x_1-y_{1})^{2}}{4(t-\tau)^2}\right]K_{-s}(x-y,t-\tau)dyd\tau.
\end{align}
Comparing \eqref{1.27} and \eqref{1.28}, the desired estimate easily follows from the inequality
\begin{align}\label{1.29}
    \left[\frac{1}{2\tau} +\frac{y_{1}^{2}}{4\tau^2}\right]K_{-s}(y,\tau)\leq CK_{-s}(y,\tau)~~~\text{ for any } (y,\tau)\in \R^{2}\times\R^{+}.
\end{align}
However, \eqref{1.29} is impossible because the coefficient $\left[\frac{1}{2\tau} +\frac{y_{1}^{2}}{4\tau^2}\right]$ is unbounded over $\R^{2}\times\R^{+}$. To resolve this issue, we introduce a new idea motivated by \cite{WGC}. Rather than directly controlling the derivatives of 
$v$ by its value at the same point, we use five nearby points $(x^{i},t), i=1,2,3,4$ and $(x,t_{-})$ to achieve the bound
\begin{align}\label{md}
    |v_{1,x_{1}x_{1}}(x,t)|\leq \sum_{i=1}^{4}v_{1}(x^{i},t)+v_{1}(x,t_{-}),~~\text{for any }(x,t)\in \widetilde{Q}_{\frac 12}(0,0).
\end{align}
Set
$$-\eta_{1}=(\frac{1}{4},\frac{1}{4}), -\eta_{2}=(-\frac{1}{4},\frac{1}{4}), -\eta_{3}=(-\frac{1}{4},-\frac{1}{4}), -\eta_{4}=(\frac{1}{4},-\frac{1}{4}).$$
For $y\in B_{1/2}^{c}(0)\cap I_{j},$ with $I_{j}~(j=1,2,3,4)$ representing the $j$-th quadrants of $\R^{2}$, it can be verified that
\begin{eqnarray}
    		\begin{aligned} 
    			|y|^2
    			&\geq|y+\eta_j|^2+\frac{1}{8}|y|+\frac{1}{16}.
    		\end{aligned}
    	\end{eqnarray}
Consequently, for any $m,l\in \R$,
\begin{eqnarray}
    	\begin{aligned}
    		\frac{|y|^{2}}{\tau^{2}}K_{-s}(y,\tau)\leq\:C\sum_{j=1}^{4}K_{-s}(y+\eta_{j},\tau),(y,\tau)\in B_{\frac{1}{2}}^c(0)\times \mathbb{R}^+.
    	\end{aligned}
 \end{eqnarray}
 
A similar perturbation can be applied to the time variable $\tau$. Hence, 
taking $$x^{j}=x+\eta_{j}~ \text{ and }t_{-}=t+1/8,$$leads to the desired estimate \eqref{md}.

 Using this perturbative approach, we establish the following fundamental estimates for the kernel $K_{-s}$.
    \begin{mlem}[Kernel estimate]\label{translation}
    	Let $r>0$ and $l,m\in \R$, satisfying $0\leq m\leq l$. There exist $2^n$ points $\eta_{j}\in \partial B_{\frac{r}{\sqrt{n}}}$ and a positive constant $C=C(n,m,l)$ such that for any $(x,t)\in \widetilde{Q}_{r}^{c}\backslash\{t=0\}$,
    	\begin{align}\label{3.2-4}
    		\frac{|x|^{m}}{|t|^{l}}K_{-s}(x,|t|)\leq Cr^{m-2l}\left(\sum_{j=1}^{2^n}K_{-s}(x+\eta_{j},|t|)+K_{-s}(x,|t|+r^{2}/n)\right).
    	\end{align}
    Furthermore, for each $k\in \N^{*}$, it holds
        \begin{align}\label{2.2-100}
            |D^{k}K_{-s}(x,|t|)|\leq Cr^{-k}\left(\sum_{j=1}^{2^n}K_{-s}(x+\eta_{j},|t|)+K_{-s}(x,|t|+r^{2}/n)\right).
        \end{align}
    \end{mlem}
\par In the next step, we show the pointwise regularity of $w_{1}$ given in \eqref{1.25-1}.
    \begin{mpro}\label{w(x,t)}   
    	For some $(x_{0},t_{0})\in\R^{n}\times\R$, assume that $f(x,t)\geq 0$ with $f\in C_{1}^{k+\alpha}(0,0;1)$ and $w_{1}$ is defined by \eqref{1.25-1}, then there exists a positive constant $C=C(n,k,\alpha,s)$ such that
    	\begin{enumerate}
    		\item[(a)] for $2s+\alpha\notin\Z$,
    		\begin{align}\label{w-6}
    			\|w_{1}\|_{C_{1}^{k+\alpha+2s}(0,0;\frac{1}{2})}\leq C\|f\|_{C_{1}^{k+\alpha}(0,0;1)};	
    		\end{align}
    		\item[(b)] for $2s+\alpha\in\Z$,\\
            when $k+\alpha+2s$ is odd,
    		\begin{align}
    			\|w_{1}\|_{C_{1}^{k+\alpha+2s,x-\ln}(0,0;\frac{1}{2})}\leq C\|f\|_{C_{1}^{k+\alpha}(0,0;1)},	
    		\end{align}
    		when $k+\alpha+2s$ is even,
    		\begin{align}
    			\|w_{1}\|_{C_{1}^{k+\alpha+2s,\ln}(0,0;\frac{1}{2})}\leq C\|f\|_{C_{1}^{k+\alpha}(0,0;1)}.	
    		\end{align}
    		\item[~] Furthermore, under the stronger condition $f\in C_{1}^{k+\alpha,Dini}(0,0;1)$, there holds
    		\begin{align}
    			\|w_{1}\|_{C_{1}^{k+\alpha+2s}(0,0;\frac{1}{2})}\leq C\|f\|_{C_{1}^{k+\alpha,Dini}(0,0;1)}.
    		\end{align}
        \end{enumerate}
    \end{mpro}
\par
\par

To estimate $w_{1}$, our aim is to construct a polynomial directly. Since the difference $w_{1}-w_{r}$ is smooth in $\widetilde{Q}_{\frac{r}{2}}(x_{0},t_{0})$, it admits a natural polynomial representation via Taylor expansion. 

To estimate the remainder term rigorously, we introduce an equivalent characterization for pointwise function spaces. Specifically, we define a series that captures the pointwise continuity of a function:
\begin{equation*}
    \sum_{i=0}^{m}r^{-i(k+\alpha)}\nu_{f}\left(r_{0}r^{i}\right),
\end{equation*}
where $\nu_{f}(R)$ are defined as \eqref{nu-f}. 
This series plays a key role in unifying the analysis across various function spaces and parameter regimes. 
For instance, if $2s+\alpha\in \Z$, we take $r_{0}=r=1/2$, 
then for any integer $m\geq 1$,
\begin{align*}
    \sum_{i=1}^{m}2^{(k+\alpha)i}\nu_{f}\left(\frac{1}{2^{i}}\right)\leq K&\Longleftrightarrow f\in C_{1}^{k+\alpha,Dini}(x_{0},t_{0};1/2),\\
    \sum_{i=1}^{m}2^{(k+\alpha)i}\nu_{f}\left(\frac{1}{2^{i}}\right)\leq Km&\Longleftrightarrow f\in C_{1}^{k+\alpha}(x_{0},t_{0};1/2),\\
    \sum_{i=1}^{m}2^{(k+\alpha)i}\nu_{f}\left(\frac{1}{2^{i}}\right)\leq Km^{2}&\Longleftrightarrow f\in C_{1}^{k,\ln}(x_{0},t_{0};1/2).
\end{align*}
Different from the method employed in \cite{LW1},
our framework allows for a unified and simplified treatment of both $f\in C_{1}^{k,\alpha}(x_{0},t_{0};1)$ and $f\in C_{1}^{k,\alpha,Dini}(x_{0},t_{0};1)$. 
In addition, we are able to derive stronger regularity results than those established in \cite{LW1} and \cite{WGC}. 



This paper is organized as follows.
In Section \ref{S2}, we first prove Lemma \ref{translation} and then establish the $C^{\infty}$ regularity of the external part $v_{1}$. In Section \ref{S3}, we obtain the pointwise regularity by decomposing the internal part $w_{1}$ into three components and analyzing them separately.
Section \ref{S4} (Appendix) gives some equivalent characterizations of the pointwise function spaces.

Throughout this paper, for simplicity, we always take $(x_{0},t_{0})=(0,0)$ and $Q_{r}(0,0)=Q_{r}, \widetilde{Q}_{r}(0,0)=\widetilde{Q}_{r}, B_{r}(0)=B_{r}$, where
$r\in(0,1]$. We denote $\lfloor \cdot \rfloor$ as the greatest integer less than or equal to its argument.  
    \par\quad    
    \section{Estimate of the External Part $v_{r}$} \label{S2}
    In this section, we establish the pointwise regularity of the external part \begin{align}\label{est-v}
    		v_{r}(x,t)=\int_{-\infty}^{t}\int_{\R^{n}}f_{\widetilde{Q}_{r}^{c}}(y,\tau)K_{-s}(x-y,t-\tau)dyd\tau.
     \end{align} 
     The subscript $r$ in $v_{r}$ may be omitted when there is no risk of confusion.
     
     We begin by establishing several useful lemmas concerning kernel estimates.
    \begin{lemma}[Global estimate for kernel function]\label{global}
	For any $A,a,b,r\in\R$ with $0\leq a\leq b$ and $ r,A> 0$, there exists a positive constant $C=C(a,b,A)>0$ such that
	\begin{align}
		\frac{|x|^{2a}}{|t|^{b}}e^{-A\frac{|x|^{2}}{|t|}}\leq Cr^{2(a-b)},~~\text{for any }(x,t)\in  \widetilde{Q}_{r}^{c}\backslash\{t=0\}.
	\end{align}
\end{lemma}
\begin{proof}
We consider two cases:

    Case 1: For $|t|\leq r^{2}$ (and thus $x\in B_{r}^{c}$), we obtain
	\begin{eqnarray}
		\begin{aligned}
			\frac{|x|^{2a}}{|t|^{b}}e^{-A\frac{|x|^{2}}{|t|}}&=|x|^{2(a-b)}\left(\frac{|x|^{2}}{|t|}\right)^{b}e^{-A\frac{|x|^{2}}{|t|}}\leq C|x|^{2(a-b)}\leq Cr^{2(a-b)}.
		\end{aligned}
	\end{eqnarray}

    Case 2: For $|t|>r^{2}$, we have
	\begin{eqnarray}
		\begin{aligned}
			\frac{|x|^{2a}}{|t|^{b}}e^{-A\frac{|x|^{2}}{|t|}}&=|t|^{a-b}\left(\frac{|x|^{2}}{|t|}\right)^{a}e^{-A\frac{|x|^{2}}{|t|}}\leq C|t|^{a-b}\leq Cr^{2(a-b)}.
		\end{aligned}
	\end{eqnarray}
    \end{proof}
    \begin{lemma}[Estimates for the derivatives of kernel function]\label{prop2.3}
        For any $k,m\in\N^{*}$, any multi-index $\alpha$ with $|\alpha|=k$ , $(x,t)\in \R^{n}\times \R^{+}$, there exists a positive constant $C=C(k,m)$ such that
        \begin{eqnarray}
    	\begin{aligned}\label{2.3-14}
    		\left|\frac{\partial^{k} K_{-s}(x,t)}{\partial x^{\alpha}}\right|\leq&CK_{-s}(x,t)\sum_{i=0}^{[\frac{k}{2}]}\frac{|x|^{k-2i}}{t^{k-i}},\\
    		\left|\frac{\partial^{m} K_{-s}(x,t)}{\partial t^{m}}\right|\leq&CK_{-s}(x,t)\sum_{i=0}^{m}\frac{|x|^{2i}}{t^{i+m}},\\
    		\left|\frac{\partial^{k+m} K_{-s}(x,t)}{\partial x^{\alpha}\partial t^{m}}\right|
    		\leq&CK_{-s}(x,t)\sum_{i=0}^{m}\sum_{j=0}^{[\frac{k}{2}]}\frac{|x|^{2i+k-2j}}{t^{k+m+i-j}}.
    	\end{aligned}
        \end{eqnarray}
        \end{lemma}
    \begin{proof}
    For illustration, we give the following direct computation:
    \begin{eqnarray}\label{L2.3-R1}
    	\begin{aligned}
    		\frac{\partial K_{-s}}{\partial {x_i}}(x,t)&= K_{-s}(x,t)\frac{x_i}{2t},\\
    		\frac{\partial K_{-s}}{\partial {t}}(x,t)&= K_{-s}(x,t)\left(\frac{|x|^{2}}{4t^{2}}-\frac{C_{1}}{t^{1}}\right),
    	\end{aligned}
    \end{eqnarray}
    and
    \begin{eqnarray}
    	\begin{aligned}\label{L2.3-R2}
    		\frac{\partial^2 K_{-s}}{\partial {x_i}\partial {x_j}}(x,t)=& K_{-s}(x,t)\left[\frac{\delta_{ij}}{2t} +\frac{x_ix_j}{4t^2}\right],\\
    		\frac{\partial^2 K_{-s}}{\partial {t}^2}(x,t)=& K_{-s}(x,t)\left[\frac{C_{1}|x|^{4}}{t^{4}}+\frac{C_{2}|x|^{2}}{t^{3}}-\frac{C_{3}}{t^{2}}\right],\\
            \frac{\partial^2 K_{-s}}{\partial {x_i}\partial {t}}(x,t)=& K_{-s}(x,t)\left[\frac{C_{1}x_{i}}{t^{2}} +\frac{C_{2}x_i|x|^{2}}{t^3}\right].
    	\end{aligned}
    \end{eqnarray}
    Hence, we conclude \eqref{2.3-14} by standard induction.
    \end{proof}
    Next, it is in a position to prove Lemma \ref{translation}, which will be useful for controlling  $D^{k}v$ by $v$ itself.
    \begin{proof}[Proof of Lemma \ref{translation}]
       For any $x\in B_{r}^{c}\cap I_{j}$ with $I_{j}~(j=1,2,\cdots,2^{n})$ representing the $j$-th quadrant of $\R^{n}$, choosing $-\eta_{j}\in I_{j}$ such that each component of $\eta_{j}$ has magnitude $r/n$,  then we have
        \begin{eqnarray}\label{ineq}
            \begin{aligned}
                -x\cdot\eta_{j}\geq \frac{r}{n}|x|.
            \end{aligned}
        \end{eqnarray}
        Consequently, 
    	\begin{eqnarray}\label{3.2-3}
    		\begin{aligned} 
    			|x|^2&= |x+\eta_{j}|^2-|\eta_{j}|^2-2\eta_{j}\cdot x\geq |x+\eta_{j}|^2-\frac{r^{2}}{n}+\frac{2r}{n}|x|\\
    			&\geq|x+\eta_{j}|^2+\frac{r}{2n}|x|+\frac{r^{2}}{2n}.
    		\end{aligned}
    	\end{eqnarray}
    	Thus, submitting (\ref{3.2-3}) into the left side of (\ref{3.2-4}) yields, for any $x\in B^c_r\cap I_j$ and $ t\neq 0,$
    	\begin{eqnarray}
    	\begin{aligned}\label{2.9-x}
    		\frac{|x|^{m}}{|t|^{l}}e^{-\frac{|x|^{2}}{4|t|}}& \leq
            \:\left(\frac{|x|^{m}}{|t|^{l}}e^{-\frac{\frac{r}{2n}|x|+\frac{r^{2}}{2n}}{4|t|}}\right)e^{-\frac{|x+\eta_{j}|^{2}}{4|t|}}\\
            &=Cr^{m-2l}\left(\left(\frac{r|x|}{8n|t|}\right)^{m}e^{-\frac{r|x|}{8n|t|}}\right)\left(\left(\frac{r^{2}}{8n|t|}\right)^{l-m}e^{-\frac{r^{2}}{8n|t|}}\right)e^{-\frac{|x+\eta_{j}|^{2}}{4|t|}}\\
    		&\leq\:Cr^{m-2l}e^{-\frac{|x+\eta_{j}|^{2}}{4|t|}}.
    	\end{aligned}
        \end{eqnarray}
    	In the alternative case, where $x\in B_{r}$ (and hence $|t|\geq r^{2}$), we observe that
    	\begin{align}\label{t-2.11}
    	-\frac{1}{|t|}\leq -\frac{1}{|t|+r^{2}/n}-\frac{r^{2}}{(n+1)|t|^{2}},	
    	\end{align}
    	and
    	\begin{align}\label{2.12-t}
    	\frac{1}{|t|^{\frac{n}{2}+1-s}}\leq\frac{C_{n}}{(|t|+r^{2}/n)^{\frac{n}{2}+1-s}}.
    	\end{align}
    	Hence, combining \eqref{t-2.11} with assumption $m\leq l$, for any $x\in B_{r}$ and $|t|\geq r^{2},$ we derive
    	\begin{eqnarray}\label{2.13-t}
    	\begin{aligned}
    		\frac{|x|^{m}}{|t|^{l}}e^{-\frac{|x|^{2}}{4|t|}} &\leq Cr^{-m}|t|^{m-l}\left(\left(\frac{r^{2}|x|^{2}}{(n+1)|t|^{2}}\right)^{\frac{m}{2}}e^{-\frac{r^{2}|x|^{2}}{4(n+1)|t|^{2}}}\right)e^{-\frac{|x|^{2}}{4(|t|+r^{2}/n)}}\\
    		&\leq Cr^{m-2l}e^{-\frac{|x|^{2}}{4(|t|+r^{2}/n)}}.
    	\end{aligned}
    \end{eqnarray}
    	Together with \eqref{2.9-x} and \eqref{2.12-t}, we obtain (\ref{3.2-4}). Then, by Lemma \ref{prop2.3}, we derive (\ref{2.2-100}). 
        \end{proof}

    Now we establish the higher regularity for $v$.
    \begin{lemma}\label{time-lemma}
    	Let $f\geq 0$ in $\R^{n}\times\R$ and let $v$ be defined in \eqref{est-v}. Then $v\in C^{\infty}(\widetilde{Q}_{\frac{r}{2}})$, and for any $k\in\N^{*}$, there exists a positive constant $C=C(n,k)$ such that
    	\begin{align}
    		\|D^{k}v\|_{L^{\infty}(\widetilde{Q}_{\frac{r}{2}})}\leq Cr^{-k}\|v\|_{L^{\infty}(\widetilde{Q}_{\frac{1+\sqrt{n}}{2\sqrt{n}}r})},
    	\end{align}
       and for any multi-index $\zeta$ and integer $m\in\N$, we have
        \begin{align}\label{v-Derivative}
    		\frac{\partial^{|\zeta|}\partial^{m} v(x,t)}{\partial x^{\zeta}\partial t^{m}}=\int_{-\infty}^{t}\int_{\mathbb{R}^{n}}f_{\widetilde{Q}_{r}^{c}}(y,\tau)\frac{\partial^{|\zeta|}\partial^{m} K_{-s}}{\partial x^{\zeta}\partial t^{m}}(x-y,t-\tau)dyd\tau.
    	\end{align}
    \end{lemma}
    \begin{proof}
        We first consider the regularity of $v$ with respect to $t$. Denote
    	\begin{align}\label{3.3-11}
    		h_{k}(x,t)=\int_{-\infty}^{t}\int_{\mathbb{R}^{n}}f_{\widetilde{Q}_{r}^{c}}(y,\tau)\frac{\partial^{k} K_{-s}}{\partial t^{k}}(x-y,t-\tau)dyd\tau.
    	\end{align}
    	Now we prove 
        $\frac{\partial^{k} v}{\partial t^{k}}=h_{k}$
        by mathematical induction. Assume that this holds for some $k\in\N$ and
        almost every $(x,t)\in \widetilde{Q}_{r/2}$, 
        then we will show that it is also true for $k+1$. 
        
    	Indeed, for any $x\in B_{r/2}$ and $\varphi\in C_0^\infty((-r^{2}/4,r^{2}/4))$ with $\varphi\geq 0$, we have
    	
    	\begin{eqnarray}\label{3.3-6}
    		\begin{aligned}
    		\int_{-\frac{r^{2}}{4}}^{\frac{r^{2}}{4}}\frac{\partial h_{k}(x,t)}{\partial t}\varphi(t)dt
    		&=-\int_{-\frac{r^{2}}{4}}^{\frac{r^{2}}{4}}h_{k}(x,t)\frac{d\varphi(t)}{dt}dt=-\int_{-\frac{r^{2}}{4}}^{\frac{r^{2}}{4}}h_{k}(x,t)\left(\lim_{\delta\to0}\frac{\varphi(t)-\varphi(t-\delta)}{\delta}\right)dt\\
    		&=-\lim_{\delta\to0}\frac{1}{\delta}\left\{\int_{-\frac{r^{2}}{4}}^{\frac{r^{2}}{4}}h_{k}(x,t)\varphi(t)dt-\int_{-\frac{r^{2}}{4}}^{\frac{r^{2}}{4}}h_{k}(x,t)\varphi(t-\delta)dt\right\}\\
    		&=-\lim_{\delta\to0}\frac{1}{\delta}\left\{\int_{-\frac{r^{2}}{4}}^{\frac{r^{2}}{4}}h_{k}(x,t)\varphi(t)dt-\int_{-\frac{r^{2}}{4}-\delta}^{\frac{r^{2}}{4}-\delta}h_{k}(x,t+\delta)\varphi(t)dt\right\}\\
    		&=-\lim_{\delta\to0}\frac{1}{\delta}\Bigg\{\int_{-\frac{r^{2}}{4}}^{\frac{r^{2}}{4}}[h_{k}(x,t)-h_{k}(x,t+\delta)]\varphi(t)dt\\
    		&\quad-\left(\int_{-\frac{r^{2}}{4}-\delta}^{-\frac{r^{2}}{4}}-\int_{\frac{r^{2}}{4}-\delta}^{\frac{r^{2}}{4}}\right)h_{k}(x,t+\delta)\varphi(t)dt\Bigg\}\\
    		&=-\lim_{\delta\to0}\int_{-\frac{r^{2}}{4}}^{\frac{r^{2}}{4}}\frac{h_{k}(x,t)-h_{k}(x,t+\delta)}{\delta}\varphi(t)dt,
    	\end{aligned}
    \end{eqnarray}
    	where the last equality follows from
    	\begin{align*}
    		\lim_{\delta\to0}\frac1\delta\int_{-\frac{r^{2}}{4}-\delta}^{-\frac{r^{2}}{4}}h_{k}(x,t+\delta)\varphi(t)dt
    		&=(-1)^{k}\lim_{\delta\to0}\frac1\delta\int_{-\frac{r^{2}}{4}-\delta}^{-\frac{r^{2}}{4}}v(x,t+\delta)\frac{\partial^{k} \varphi}{\partial t^{k} }(t)dt=0,\\
    		\lim_{\delta\to0}\frac1\delta\int_{\frac{r^{2}}{4}-\delta}^{\frac{r^{2}}{4}}h_{k}(x,t+\delta)\varphi(t)dt
    		&=(-1)^{k}\lim_{\delta\to0}\frac1\delta\int_{\frac{r^{2}}{4}-\delta}^{\frac{r^{2}}{4}}v(x,t+\delta)\frac{\partial^{k} \varphi}{\partial t^{k} }(t)dt=0.\\ 
    	\end{align*}
    	Next, we write 
        \begin{equation}\label{2.21}
        \int_{-\frac{r^{2}}{4}}^{\frac{r^{2}}{4}}\frac{h_{k}(x,t+\delta)-h_{k}(x,t)}{\delta}\varphi(t)dt:=J_{1}+J_{2},
        \end{equation}
        where
    	\begin{eqnarray}\label{3.3-4}
    		\begin{aligned}
    		&J_{1}=\int_{-\frac{r^{2}}{4}}^{\frac{r^{2}}{4}}\int_{-\infty}^t\int_{\mathbb{R}^n}f_{\widetilde{Q}_{r}^{c}}(y,\tau)\frac{1}{\delta}\left(\frac{\partial^{k} K_{-s}}{\partial t^{k}}(x-y,t+\delta-\tau)-\frac{\partial^{k} K_{-s}}{\partial t^{k}}(x-y,t-\tau)\right)dyd\tau\varphi(t)dt,\\
    		&J_{2}=\frac1\delta\int_{-\frac{r^{2}}{4}}^{\frac{r^{2}}{4}}\int_t^{t+\delta}\int_{\mathbb{R}^n}f_{\widetilde{Q}_{r}^{c}}(y,\tau)\frac{\partial^{k} K_{-s}}{\partial t^{k}}(x-y,t+\delta-\tau)dyd\tau\varphi(t)dt.
    		\end{aligned}\end{eqnarray}

        Note that as $\delta\to 0$, 
    	\begin{eqnarray}\label{3.3-5}
    	\begin{aligned}
    		J_1&\to\int_{-\frac{r^{2}}{4}}^{\frac{r^{2}}{4}}\int_{-\infty}^t\int_{\mathbb{R}^n}f_{\widetilde{Q}_{r}^{c}}(y,\tau)\frac{\partial^{k+1} K_{-s}}{\partial t^{k+1}}(x-y,t-\tau)dyd\tau\varphi(t)dt=\int_{-\frac{r^{2}}{4}}^{\frac{r^{2}}{4}}h_{k+1}(x,t)\varphi(t)dt.
    	\end{aligned}
        \end{eqnarray}
    	In the next step, we prove that
    	\begin{align}\label{3.3-2}
    		J_2\to 0~~\:\mathrm{as}\:\delta\to0.
    	\end{align}
        By Lemma \ref{prop2.3}, for sufficiently small $\delta>0$ and any $\tau\in(0,\delta)$,
        \begin{eqnarray}
        \begin{aligned}
            \left|\frac{\partial^{k} K_{-s}(y,\tau)}{\partial t^{k}}\right|\leq CK_{-s}(y,\tau)\sum_{i=0}^{k}\frac{|y|^{2i}}{\tau^{i+k}}
            \leq CK_{-s}(y,\delta)\sum_{i=0}^{k}\frac{|y|^{2i}}{\delta^{i+k}}.
        \end{aligned}
        \end{eqnarray}
    	Thus by Lemma \ref{global},
        \begin{eqnarray}\label{3.3-3}
    	\begin{aligned}
    		J_{2}& =\frac{1}{\delta}\int_{-\frac{r^{2}}{4}}^{\frac{r^{2}}{4}}\int_{t}^{t+\delta}\int_{\mathbb{R}^{n}}f_{\widetilde{Q}_{r}^{c}}(y,\tau)\frac{\partial^{k} K_{-s}}{\partial t^{k}}(x-y,t+\delta-\tau)dyd\tau\varphi(t)dt  \\
    		&=\frac{1}{\delta}\int_{-\frac{r^{2}}{4}}^{\frac{r^{2}}{4}}\int_{0}^{\delta}\int_{\mathbb{R}^{n}}f_{\widetilde{Q}_{r}^{c}(0,-t)}(y,\tau+t)\frac{\partial^{k} K_{-s}}{\partial t^{k}}(x-y,\delta-\tau)dyd\tau\varphi(t)dt \\
    		&\leq C\sum_{i=0}^{k}\frac{1}{\delta}\int_{-\frac{r^{2}}{4}}^{\frac{r^{2}}{4}}\int_{0}^{\delta}\int_{\mathbb{R}^{n}}f_{\widetilde{Q}_{r}^{c}(0,-t)}(y,\tau+t)K_{-s}(x-y,\delta)\frac{|x-y|^{2i}}{\delta^{i+k}}dyd\tau\varphi(t)dt\\
            &\leq Cr^{-2k}\frac{1}{\delta}\int_{-\frac{r^{2}}{4}}^{\frac{r^{2}}{4}}\int_{0}^{\delta}\int_{\mathbb{R}^{n}}f_{\widetilde{Q}_{r}^{c}(0,-t)}(y,\tau+t)\frac{e^{-\frac{|x-y|^{2}}{8\delta}}}{\delta^{\frac{n}{2}+1-s}}dyd\tau\varphi(t)dt.
            \end{aligned}
        \end{eqnarray}
        Denote
        \begin{align}
            J_{3}:=\frac{1}{\delta}\int_{-\frac{r^{2}}{4}}^{\frac{r^{2}}{4}}\int_{0}^{\delta}\int_{\mathbb{R}^{n}}f_{\widetilde{Q}_{r}^{c}(0,-t)}(y,\tau+t)K_{-s}(x-y,2\delta)dyd\tau\varphi(t)dt.
        \end{align}
        Then, for $\delta$ sufficiently small,
            \begin{eqnarray}
    	\begin{aligned}\label{2.28}
    		J_{3}&=\frac{1}{\delta}\int_{0}^{\delta}\int_{\mathbb{R}^{n}}\int_{-\frac{r^{2}}{4}}^{\frac{r^{2}}{4}}f_{\widetilde{Q}_{r}^{c}(0,-t)}(y,\tau+t)K_{-s}(x-y,2\delta)\varphi(t)dtdyd\tau  \\
    		&=\frac{1}{\delta}\int_{0}^{\delta}\int_{\mathbb{R}^{n}}\int_{-\frac{r^{2}}{4}+\tau}^{\frac{r^{2}}{4}+\tau}f_{\widetilde{Q}_{r}^{c}(0,\tau-t)}(y,t)K_{-s}(x-y,2\delta)\varphi(t-\tau)dtdyd\tau  \\
    		&\leq\sup_{t\in(-\frac{r^{2}}{4},\frac{r^{2}}{4})}\varphi(t)\frac{1}{\delta}\int_{0}^{\delta}\int_{\mathbb{R}^{n}}\int_{-\frac{r^{2}}{4}}^{\frac{3r^{2}}{8}}f_{\widetilde{Q}_{r-\delta}^{c}(0,-t)}(y,t)K_{-s}(x-y,2\delta)dtdyd\tau  \\
    		&\leq C\int_{\mathbb{R}^{n}}\int_{-\frac{r^{2}}{4}}^{\frac{3r^{2}}{8}}f_{\widetilde{Q}_{\frac{r}{2}}^{c}}(y,t)K_{-s}(x-y,2\delta)dtdy
    		:=J_{4}.
    	\end{aligned}
        \end{eqnarray} 
    	We just need to show that
    	\begin{align}\label{3.3-1}
    		J_{4}\to0 ~~\text{as}~\delta\to0.
    	\end{align}
    	When $\delta$ is small enough, there exists a $t_{0}\in (\frac{3r^{2}}{8},\frac{r^{2}}{2})$ such that for all  $\tau\in(-\frac{r^{2}}{4},\frac{3r^{2}}{8})$, hold that
    	$$f_{\widetilde{Q}_{\frac{r}{2}}^{c}}(y,t) K_{-s}(x-y,2\delta)\leq f_{\widetilde{Q}_{\frac{r}{2}}^{c}}(y,t) K_{-s}(x-y,t_{0}-t).$$
    	Consequently,
    	\begin{eqnarray}\label{3.27}
    	\begin{aligned}
    		J_{4}
    		\leq& C\int_{\mathbb{R}^{n}}\int_{-\frac{r^{2}}{4}}^{\frac{3r^{2}}{8}}f_{\widetilde{Q}_{\frac{r}{2}}^{c}}(y,t)K_{-s}(x-y,t_{0}-t)dtdy\\
    		\leq&C\int_{\mathbb{R}^{n}}\int_{-\infty}^{t_{0}}f_{\widetilde{Q}_{\frac{r}{2}}^{c}}(y,t)K_{-s}(x-y,t_{0}-t)dtdy
    		\leq \|v_{\frac{r}{2}}(x,\cdot)\|_{L^{\infty}((-\frac{r^{2}}{2},\frac{r^{2}}{2}))}.
    	\end{aligned} 
        \end{eqnarray} 
    	In addition, since
    	$$f_{\widetilde{Q}_{\frac{r}{2}}^{c}}(y,t)K_{-s}(x-y,2\delta)=f_{\widetilde{Q}_{\frac{r}{2}}^{c}}(y,t)\frac{e^{-\frac{|x-y|^{2}}{8\delta}}}{\delta^{\frac{n}{2}+1-s}}\rightarrow 0~a.e. \text{ for } (y,t)\in \R^{n}\times(-\frac{r^{2}}{4},\frac{3r^{2}}{8}),$$
    	then the validity of (\ref{3.3-1}) follows from Lebesgue's Dominated Convergence Theorem.

        Together with (\ref{3.3-3}), \eqref{2.28} and (\ref{3.3-1}), we get (\ref{3.3-2}).
    	Moreover, by synthesizing results  (\ref{2.21})-(\ref{3.3-2}), we arrive at
    	\begin{align}
    		\int_{-\frac{r^{2}}{4}}^{\frac{r^{2}}{4}}\frac{\partial h_{k}(x,t)}{\partial t}\varphi(t)dt=\int_{-\frac{r^{2}}{4}}^{\frac{r^{2}}{4}}\int_{-\infty}^t\int_{\mathbb{R}^n}f_{\widetilde{Q}_{r}^{c}}(y,\tau)\frac{\partial^{k+1} K_{-s}}{\partial t^{k+1}}(x-y,t-\tau)dyd\tau\varphi(t)dt,
    	\end{align}
       which establishes that for almost every $(x,t) \in \widetilde{Q}_{r/2}$,
        \begin{align}\label{3.3-8}
        	\frac{\partial^{k+1} v}{\partial t^{k+1}}(x,t)=\frac{\partial h_{k}}{\partial t}(x,t)=\int_{-\infty}^t\int_{\mathbb{R}^n}f_{\widetilde{Q}_{r}^{c}}(y,\tau)\frac{\partial^{k+1} K_{-s}}{\partial t^{k+1}}(x-y,t-\tau)dyd\tau=h_{k+1}(x,t).
        \end{align}
        Thus, we complete the induction.
        
        Furthermore, based on Lemma \ref{translation}, we can find $\eta_{j}\in B_{r/(2\sqrt{n})}$ such that
      \begin{eqnarray}
        \begin{aligned}\label{3.3-7}
        	\left|\frac{\partial^{k} v}{\partial t^{k}}(x,t)\right|
        	\leq&C\int_{-\infty}^{t}\int_{\R^{n}}f_{\widetilde{Q}_{r}^{c}}(y,\tau)|D^{2k}K_{-s}(x-y,t-\tau)|dyd\tau\\
        	\leq&C\int_{-\infty}^{t}\int_{\R^{n}}f_{\widetilde{Q}_{r}^{c}}(y,\tau)r^{-2k}\Bigg(\sum_{j=1}^{2^n}K_{-s}(x-y+\eta_{j},t-\tau)\\
            &\quad+K_{-s}\left(x-y,t-\tau+\frac{r^{2}}{4n}\right)\Bigg)dyd\tau
        	\leq Cr^{-2k}\|v\|_{L^{\infty}(\widetilde{Q}_{\frac{1+\sqrt{n}}{2\sqrt{n}}r})}.
        \end{aligned}
        \end{eqnarray} 

        Therefore, for any $x\in B_{\frac{r}{2}}$, $v(x,\cdot)\in W^{k,\infty}((-\frac{r^{2}}{4},\frac{r^{2}}{4}))$, by the Sobolev Embedding
    	Theorem, we deduce that 
    	$v(x,\cdot)\in C^{k-1}((-\frac{r^{2}}{4},\frac{r^{2}}{4}))$, thus
    	we have proved $v$ is smooth with respect to $t$.
        
        By a similar discussion, we can establish the $C^{\infty}$ regularity of $v$ with respect to $x$. Hence, we complete the proof of this lemma.
    \end{proof}
    Now, Proposition \ref{v-E} follows from Lemma \ref{time-lemma}.
    For $f\in L^{1}$, a similar argument to the proof of Lemma \ref{time-lemma} yields the following analogous lemma—--and hence Theorem \ref{v-E2}.
    \begin{lemma}\label{3.7}
    	Let $f\in L^{1}(\R^{n}\times \R)$ and $v$ be defined by \eqref{est-v}. Then $v\in C^{\infty}(\widetilde{Q}_{\frac{r}{2}})$ satisfies \eqref{v-Derivative}, and for any $k\in \N$, there exists a positive constant $C=C(n,k,s)$ such that
    		\begin{align}
    			\|D^{k}v\|_{L^{\infty}(\widetilde{Q}_{\frac{r}{2}})}\leq Cr^{-k}\|f\|_{L^{1}(\R^{n}\times \R)}.
    		\end{align}	
    \end{lemma}
    \begin{theorem}\label{v-E2}
    Let $f\in L^{1}(\R^{n}\times \R)$ and let $v_{1}$ be defined by \eqref{est-v} for $r=1$. Then $v_{1}\in C^{\infty}(0,0;1/2)$, and for any $k\in \N$, there exists a positive constant $C=C(n,k,s)$ such that
    	\begin{align}
    		\|v_{1}\|_{C^{k}(0,0;\frac{1}{2})}\leq C\|f\|_{L^{1}(\R^{n}\times \R)}.	
    	\end{align}
    \end{theorem}
    \section{Estimate of the Internal Part $w_{1}$} \label{S3}
   In this section, we establish the pointwise Schauder regularity for $w_{1}$, defined in \eqref{1.25-1}. It is observed that for $(x,t)\in Q_{1}$,
    \begin{align}\label{1.25}
    		w_{1}(x,t)=\int_{-\infty}^{t}\int_{\R^{n}}f_{Q_{1}}(y,\tau)K_{-s}(x-y,t-\tau)dyd\tau,
    \end{align}
   and similarly for $v_{1}$.
   
   First, we introduce a $C_0^\infty(\R^{n+1})$ cut-off function $\psi:\R^{n+1}\rightarrow[0,1]$ satisfying
    		$$\psi(x,t)=
    		\begin{cases}
    			0,~~~(x,t)\in \widetilde{Q}_{2}^{c},\\
    			1,~~~(x,t)\in \widetilde{Q}_{1}.
    		\end{cases}$$   	      
    Then for any $0<r\leq 1$, we decompose $w_{1}$ into three components.
    \begin{definition}\label{DEF}
    	Let $f\in C_{1}^{k+\alpha}(0,0;1)$ (or $ C_{1}^{k+\alpha,Dini}(0,0;1)$) with polynomial $P$, for some $k\in\N$ and $\alpha\in[0,1)$. Then we define $J=P\psi$ and decompose $w_{1}$ in $Q_{1}$:
        \begin{align}
    	w_{1}(x,t)=S_{r}(x,t)+T_{r}(x,t)+u_{P}(x,t),
        \end{align}
        where
    	\begin{align}
    			S_{r}(x,t)&=\int_{-\infty}^{t}\int_{\R^{n}}(f-J)_{Q_{r}}(y,\tau)K_{-s}(x-y,t-\tau)dyd\tau,\label{def-S}\\
    		T_{r}(x,t)&=\int_{-\infty}^{t}\int_{\R^{n}}(f-J)_{Q_{1}\backslash Q_{r}}(y,\tau)K_{-s}(x-y,t-\tau)dyd\tau, \label{def-T}
    	\end{align}
        and 
    	\begin{align}\label{def-u}
    		u_{P}(x,t)=\int_{-\infty}^{t}\int_{\R^{n}}J_{Q_{1}}(y,\tau)K_{-s}(x-y,t-\tau)dyd\tau.
    	\end{align}
    \end{definition} 
    \subsection{Estimate for $S_{r}$}
    We first present several auxiliary lemmas that will be essential in our analysis for $S_{r}$ .
    \begin{lemma}[Local estimate for kernel function]\label{local}
    	For any $a,b\geq 0$, $A>0$, there exists a positive constant $C=C(a,b,A)$ such that
    	\begin{align}\label{3.6}
    		\frac{|x|^{2a}}{|t|^{b}}e^{-A\frac{|x|^{2}}{|t|}}\leq Cr^{2(a-b)},~~ \text{ for any $(x,t)\in  \widetilde{Q}_{2r}\backslash (\widetilde{Q}_{r}\cup\{t=0\})$.}
    	\end{align}
    \end{lemma}
    \begin{proof}
    Take $(x,t)\in  \widetilde{Q}_{2r}\backslash \widetilde{Q}_{r}$. We establish \eqref{3.6} by considering the following two cases:

        Case 1: For $|t|\leq r^{2}$ (and thus $r\leq|x|\leq2r$), we have
    	\begin{eqnarray}
    		\begin{aligned}
    			\frac{|x|^{2a}}{|t|^{b}}e^{-A\frac{|x|^{2}}{|t|}}=|x|^{2(a-b)}(\frac{|x|^{2}}{|t|})^{b}e^{-A\frac{|x|^{2}}{|t|}}
    		\leq C|x|^{2(a-b)}\leq Cr^{2(a-b)}.
    		\end{aligned}
    	\end{eqnarray}

        Case 2: For $r^{2}<|t|\leq 4r^{2}$, we obtain
    	\begin{eqnarray}
    		\begin{aligned}
    			\frac{|x|^{2a}}{|t|^{b}}e^{-A\frac{|x|^{2}}{|t|}}=|t|^{a-b}(\frac{|x|^{2}}{|t|})^{a}e^{-A\frac{|x|^{2}}{|t|}}\leq C|t|^{a-b}\leq Cr^{2(a-b)}.
    		\end{aligned}
    	\end{eqnarray}
    \end{proof}
    \begin{lemma}[Decay estimate for $S_{r}$]\label{ref-w()}
    	Let $f\in C_{1}^{k+\alpha}(0,0;1)$ with polynomial $P$, for some $k\in\N$ and $\alpha\in[0,1)$. For $S_{r}$ defined by \eqref{def-S}, there exists a positive constant $C=C(n,k,s)$ such that
    	\begin{align}
    		\bbint_{Q_{r}}|S_{r}(y,\tau)|dyd\tau\leq C\|f\|_{C_{1}^{k+\alpha}(0,0;1)}r^{k+\alpha+2s}.
    	\end{align}
    \end{lemma}
    \begin{proof}
        Define
    	$$g(x,t)=\frac{f(x,t)-P(x,t)}{(|x|^{2}+|t|)^{\frac{k+\alpha}{2}}},$$
    	then we obtain
    	\begin{eqnarray}
    	\begin{aligned}
    		|S_{r}(x,t)|\leq&\int_{B_{r}}\int_{-r^{2}}^{t}|f-P|(y,\tau)K_{-s}(x-y,t-\tau)dyd\tau\\
            =&\int_{B_{r}(x)}\int_{0}^{t+r^{2}}|f-P|(x-y,t-\tau)K_{-s}(y,\tau)dyd\tau\\
    		=&\int_{B_{r}(x)}\int_{0}^{t+r^{2}}(|x-y|^{2}+|t-\tau|)^{\frac{k+\alpha}{2}}|g(x-y,t-\tau)|K_{-s}(y,\tau)dyd\tau\\ 
    		\leq&Cr^{k+\alpha}\int_{B_{r}(x)}\int_{0}^{t+r^{2}}|g(x-y,t-\tau)|\frac{e^{-\frac{|y|^{2}}{4\tau}}}{|\tau|^{\frac{1}{2}(n+2)-s}}dyd\tau.
            \end{aligned}
            \end{eqnarray}
         Define
         \begin{align*}           
\nu_{g}(R):=\sup_{0<r\leq R}\bbint_{-r^{2}}^{0}\bbint_{ B_{r}}|g(x,t)|dxdt, 
\end{align*} 
        by Lemma \ref{local} with $a=0,b=n/2+2-s$ and $A=4$, we have
            \begin{eqnarray}\label{rad-1}
    	\begin{aligned}
        |S_{r}(x,t)|\leq&Cr^{k+\alpha}\sum_{i=0}^{\infty}\int_{\widetilde{Q}_{2^{1-i}r}\backslash \widetilde{Q}_{2^{-i}r}}|g(x-y,t-\tau)|\frac{e^{-\frac{|y|^{2}}{4\tau}}}{|\tau|^{\frac{1}{2}(n+2)-s}}\chi_{B_{r}(x)\times[0,t+r^{2})}(y,\tau)dyd\tau\\
 \leq&Cr^{k+\alpha}\sum_{i=0}^{\infty}\left(\frac{r}{2^{i}}\right)^{2s-n-2}\int_{\widetilde{Q}_{2^{1-i}r}\backslash \widetilde{Q}_{2^{-i}r}}|g(x-y,t-\tau)|\chi_{B_{r}(x)\times[0,t+r^{2})}(y,\tau)dyd\tau\\
            \leq& Cr^{k+\alpha+2s}\sum_{i=0}^{\infty}\frac{1}{2^{2si}}\left(\left(\frac{r}{2^{i}}\right)^{-n-2}I_{1}(x,t)\right),
    	\end{aligned}
    \end{eqnarray}
    where
    \begin{equation}
        I_{1}(x,t)=\int_{(Q_{2^{1-i}r}(x,t)\backslash Q_{2^{-i}r}(x,t))\cap Q_{1}}|g(y,\tau)|dyd\tau.
    \end{equation}
    Moreover, it holds
    \begin{eqnarray}\label{rad-2}
        \begin{aligned}
            \int_{Q_{r}}I_{1}(y,\tau)dyd\tau\leq |Q_{2^{1-i}r}|\int_{Q_{2r}\cap Q_{1}}|g(y,\tau)|dyd\tau\leq C\left(\frac{r}{2^{i}}\right)^{n+2}\left(r^{n+2}\nu_{g}(1)\right).
        \end{aligned}
    \end{eqnarray}
    By \eqref{rad-1}, \eqref{rad-2} and Theorem \ref{reduce 1}, we have
    \begin{eqnarray}
        \begin{aligned}
            \bbint_{Q_{r}}|S_{r}(y,\tau)|dyd\tau\leq Cr^{k+\alpha+2s}\nu_{g}(1)\sum_{i=0}^{\infty}\frac{1}{2^{2si}}=C\|f\|_{C_{1}^{k+\alpha}(0,0;1)}r^{k+\alpha+2s}.
        \end{aligned}
    \end{eqnarray}
   Hence, we complete the proof.
    \end{proof}
    \subsection{Estimate for $u_{P}$}
    In this subsection, we establish the pointwise regularity of $u_{P}$ by considering the following two functions:
    \begin{eqnarray}\label{def-W}
    		\begin{aligned}\label{3.0}
    			W_{P,r}(x,t)&=\int_{-\infty}^{t}\int_{\R^{n}}J_{Q^{c}_{r}}(y,\tau)K_{-s}(x-y,t-\tau)dyd\tau,\\
    			V_{P}(x,t)&=\int_{-\infty}^{t}\int_{\R^{n}}J(y,\tau)K_{-s}(x-y,t-\tau)dyd\tau=W_{P,1}(x,t)+u_{P}(x,t).
    		\end{aligned}
    \end{eqnarray}
    For simplicity, we write $W_{P}=W_{P,1}$ in the remainder of this section.
    \begin{theorem}\label{3.10}
    	For any polynomial $P$, let $u_{P}$ be defined by \eqref{def-u}.
    	Then $u_{P}\in C^{\infty}(0,0;1/2)$ and for any $k\in \N$, there exists a positive constant $C=C(n,k,s)>0$ such that
    		\begin{align}\label{3.10-1}
    			\|u_{P}\|_{C^{k}(0,0;1/2)}\leq C\|P\|_{(0,0)}.
    		\end{align}
    \end{theorem}
    \begin{proof}
    We consider the pointwise regularity of $W_{P}$ and $V_{P}$ as defined in \eqref{def-W} to establish \eqref{3.10-1}.

    For function $W_{P}$, since $J\in L^{1}(\R^{n}\times \R)$, by Theorem \ref{v-E2} with $f=J$, we have $W_{P}\in C^{\infty}(0,0;1/2)$, and for any $k\in\N$, 
    \begin{align}\label{4.3.9}
    			\|W_{P}\|_{C^{k}(0,0;1/2)}\leq C\|P\|_{(0,0)}. \end{align}
    Moreover, we observe that 
    	\begin{align}\label{J(x,t)}
    		\|J\|_{C^{k}(\R^{n})}\leq C\|P\|_{(0,0)},~~ \text{for any $k\in\N$,}
    	\end{align}
        since $supp(J)\subset \widetilde{Q}_{2}$.
        
    	Next, we estimate the derivatives of $V_P$. For each multi-index $\alpha$ and $m\in\N$, we have, for any $(x,t)\in \widetilde{Q}_{1}$,
    	\begin{eqnarray}
    		\begin{aligned}
    			\left|\frac{\partial^{|\alpha|+m} V_{P}}{\partial x^{\alpha}t^{m}}(x,t)\right|=&\left|\int_{0}^{\infty}\int_{\R^{n}}\frac{\partial^{|\alpha|+m} J}{\partial x^{\alpha}t^{m}}(x-y,t-\tau)K_{-s}(y,\tau)dyd\tau\right|\\
    			\leq&C\|P\|_{(0,0)}\int_{\R^{n}\times\R}K_{-s}(y,\tau)dyd\tau
    			= C\|P\|_{(0,0)}.
    		\end{aligned}
    	\end{eqnarray}
    	This establishes the uniform bound
    	\begin{align}\label{2.1-3}
    		\|D^{k}V_{P}\|_{L^{\infty}(\widetilde{Q}_{1})}\leq C\|P\|_{(0,0)},~~ \text{for any $k\in\N$,}
    	\end{align}
    	which implies $V_{P}\in C^{\infty}(\widetilde{Q}_{1})$. In particular, we obtain, for any $k\in \N$,
    \begin{align}\label{2.1-2}
    			\|V_{P}\|_{C^{k}(0,0;1/2)}\leq C\|P\|_{(0,0)}.	
    		\end{align}
            Hence, together with \eqref{def-W}, \eqref{4.3.9} and \eqref{2.1-2}, we get \eqref{3.10-1}.
        \end{proof}
    \subsection{Estimate for $T_{r}$ with $2s+\alpha\notin\Z$}
    In this subsection, we obtain the pointwise regularity of $T_{r}$ by constructing a sequence of parabolic polynomials belonging to the function space $\mathcal{P}_{k}$ defined by Definition \ref{poly}.
    \begin{lemma}\label{lemma 3.8}
    For $\alpha\in[0,1)$ with $2s+\alpha\notin\Z$, let $f\in C_{1}^{k+\alpha}(0,0;1)$($k\in\N$) be positive, and suppose that $T_{r}$ is given by \eqref{def-T}. Then for any $\eta\in[1/4,1/2]$, there exist a sequence of polynomials $P_{i}\in\mathcal{P}_{\gamma}$, where $\gamma = k + \lfloor \alpha + 2s \rfloor$, and a positive constant $C=C(n,k,\alpha,s)$ such that for every integer $i\geq 1$, 
    		\begin{align}\label{4.9-1}
    			\|T_{\eta^{i-1}}-P_i\|_{L^{\infty}(Q_{\eta^{i}})}\leq C\|f\|_{C_{1}^{k+\alpha}(0,0;1)}\eta^{i(k+\alpha+2s)}.
    		\end{align}
    	In addition, there exist constants $a_j$ ($j = 0,\ldots,\gamma$) such that for every $i\geq1$,
    		\begin{align}\label{4.9-3}
    			\sum_{j=0}^{\gamma}\eta^{j(i-1)}|D^{j}P_{i}(0,0)-a_{j}|\leq C\|f\|_{C_{1}^{k+\alpha}(0,0;1)}\eta^{i(k+\alpha+2s)}.
    		\end{align} 		
    \end{lemma}
    \begin{proof}
        We set $P_{1}(x,t)=0$. Now we fix an $i\geq 2$.

        Recall that $v_{r}$ is defined in \eqref{est-v} and $W_{P,r}$ in \eqref{def-W}. Since $$T_{r}=v_{r}-v_{1}+W_{P}-W_{P,r},$$
    	then from Lemma \ref{time-lemma}, Lemma \ref{3.7} and Lemma \ref{prop2.3}, there exists a positive constant $C=C(|\zeta|,\beta,n,s)$ such that for any multi-indices $\zeta$ and integer $\beta$, 
        \begin{eqnarray}\label{der of T_{r}}
            \begin{aligned}
    		\left|\frac{\partial^{|\zeta|+\beta} T_{r}}{\partial x^{\zeta}\partial t^{\beta}}(x,t)\right|&=\left|\int_{-\infty}^{t}\int_{\mathbb{R}^{n}}\left(f-P\right)_{Q_{1}\backslash Q_{r}}(y,\tau)\frac{\partial^{|\zeta|+\beta}K_{-s}}{\partial x^{\zeta}\partial t^{\beta}}(x-y,t-\tau)dyd\tau\right|\\
    		&\leq C\sum_{l=0}^{\beta}\sum_{j=0}^{[\frac{|\zeta|}{2}]}\int_{-\infty}^{t}\int_{\mathbb{R}^{n}}|f-P|_{Q_{1}\backslash Q_{r}}(y,\tau)\frac{|x-y|^{2l+|\zeta|-2j}e^{-\frac{|x-y|^{2}}{4(t-\tau)}}}{(t-\tau)^{\frac{n}{2}+1-s+|\zeta|+\beta+l-j}}dyd\tau,
    	\end{aligned}
        \end{eqnarray}
        for any $(x,t)\in Q_{r/2}$,
    	which implies $T_{r}\in C^{\infty}(\widetilde{Q}_{r/2})$. Hence, Taylor's expansion of $T_{r}$ with $r=\eta^{i-1}$ yields that	
    	\begin{eqnarray}\begin{aligned}\label{exp-3.22}
    		T_{\eta^{i-1}}(x,t)=P_{i}(x,t)+R_{i}(x_{\xi},t_{\xi}), \text{ for any }(x,t)\in Q_{\eta^{i}},
    	\end{aligned}\end{eqnarray}
        where 
        \begin{eqnarray}
            \begin{aligned}\label{3.23}
              P_{i}(x,t)&:=\sum_{j=0}^{\gamma}\sum_{|\sigma|=j}\frac{1}{\sigma!}\left(\frac{\partial^{\sigma} T_{\eta^{i-1}}}{\partial (x,t)^{\sigma}}(0,0)\right)\cdot(x,t)^{\sigma},\\
              R_{i}(x,t) &:=\sum_{|\sigma|=\gamma+1}\left(\frac{1}{\sigma!}\frac{\partial^{\sigma} T_{\eta^{i-1}}}{\partial (x,t)^{\sigma}}(x,t)\right)\cdot(x,t)^{\sigma},\\ 
            \end{aligned}
        \end{eqnarray}
    	and $(x_{\xi},t_{\xi})\in Q_{\eta^i}$ lies on the line segment between $(0,0)$ and $(x,t)$. 

        By \eqref{der of T_{r}}, we get
    	\begin{eqnarray}\label{4.9-4}
    	\begin{aligned}
    		|R_{i}(x_{\xi},t_{\xi})| 
    		\leq C\eta^{i(\gamma+1)}\sum_{|\zeta|+2\beta=\gamma+1}\sum_{l=0}^{\beta}\sum_{j=0}^{\lfloor\frac{|\zeta|}{2}\rfloor}I_{i,j,l,|\zeta|}, 
    	\end{aligned}
        \end{eqnarray}
        where
        \begin{align}\label{I}
        I_{i,j,l,|\zeta|}:=\int_{-\infty}^{t_{\xi}}\int_{\mathbb{R}^{n}}|f-P|_{Q_{1}\backslash Q_{\eta^{i-1}}}(y,\tau)\frac{|x_{\xi}-y|^{2l+|\zeta|-2j}e^{-\frac{|x_{\xi}-y|^{2}}{4(t_{\xi}-\tau)}}}{(t_{\xi}-\tau)^{\frac{n}{2}+1-s+|\zeta|+\beta+l-j}}dyd\tau.
        \end{align}
        By Definition \ref{sec:space-a}, we have
        \begin{align*}           
\nu_{f}(R)=\sup_{0<r\leq R}\bbint_{-r^{2}}^{0}\bbint_{ B_{r}}|f-P|(y,\tau)dyd\tau.      
\end{align*}
    	By Lemma \ref{global} with $a=b=|\zeta|/2+l-j$ and $A=8$, for $(x_{\xi},t_{\xi})\in Q_{\eta^{i}}$ and $(y,\tau)\in Q_{\eta^{i-1}}^{c}$, we obtain
    	\begin{eqnarray}\label{4.9-5}
    	\begin{aligned}
    		I_{i,j,l,|\zeta|}&=\int_{-\infty}^{t_{\xi}}\int_{\mathbb{R}^{n}}\frac{|f-P|_{Q_{1}\backslash Q_{\eta^{i-1}}}(y,\tau)e^{-\frac{|x_{\xi}-y|^{2}}{8(t_{\xi}-\tau)}}}{(t_{\xi}-\tau)^{\frac{n}{2}+1-s+\frac{|\zeta|}{2}+\beta}}\left(\frac{|x_{\xi}-y|^{2l+|\zeta|-2j}}{(t_{\xi}-\tau)^{\frac{|\zeta|}{2}+l-j}}e^{-\frac{|x_{\xi}-y|^{2}}{8(t_{\xi}-\tau)}}\right)dyd\tau\\
            &\leq C\int_{-\infty}^{t_{\xi}}\int_{\mathbb{R}^{n}}\frac{|f-P|_{Q_{1}\backslash Q_{\eta^{i-1}}}(y,\tau)e^{-\frac{|x_{\xi}-y|^{2}}{8(t_{\xi}-\tau)}}}{(t_{\xi}-\tau)^{\frac{n}{2}+1-s+\frac{|\zeta|}{2}+\beta}}dyd\tau\\
            \end{aligned}
       \end{eqnarray}
       Recall that $(x_{\xi},t_{\xi})\in Q_{\eta^{i}}$, we observe that for $(y,\tau)\in Q_{1}\backslash Q_{\eta^{i-1}}$,
       \begin{align*}
           \eta^{2i}\leq|y|^{2}+|\tau|&\leq 2(|y-x_{\xi}|^{2}+|\tau-t_{\xi}|+|x_{\xi}|^{2}+|t_{\xi}|)\\
           &\leq2(|y-x_{\xi}|^{2}+|\tau-t_{\xi}|+\eta^{2i})\\
           &\leq 3(|y-x_{\xi}|^{2}+|\tau-t_{\xi}|).
       \end{align*}
       Hence, by Lemma \ref{global} with $a=b=(k+\alpha)/2,A=16$, we have
        \begin{eqnarray}\label{4.95}
    	\begin{aligned}
            I_{i,j,l,|\zeta|}&\leq  C\sum_{j=1}^{i-1}\eta^{-j(k+\alpha)}\int_{-\infty}^{t_{\xi}}\int_{\mathbb{R}^{n}}|f-P|_{Q_{\eta^{j-1}}\backslash Q_{\eta^{j}}}(y,\tau)\frac{(|x_{\xi}-y|^{2}+|t_{\xi}-\tau|)^{\frac{k+\alpha}{2}}e^{-\frac{|x_{\xi}-y|^{2}}{8|t_{\xi}-\tau|}}}{|t_{\xi}-\tau|^{\frac{n}{2}+1-s+\frac{|\zeta|}{2}+\beta}}dyd\tau\\
    		&\leq C\sum_{j=1}^{i-1}\eta^{-j(k+\alpha)}\int_{-\infty}^{t_{\xi}}\int_{\mathbb{R}^{n}}|f-P|_{Q_{\eta^{j-1}}\backslash Q_{\eta^{j}}}(y,\tau)\frac{e^{-\frac{|x_{\xi}-y|^{2}}{16|t_{\xi}-\tau|}}}{|t_{\xi}-\tau|^{\frac{n}{2}+1+\frac{1-\{\alpha+2s\}}{2}}}dyd\tau,\\
    	\end{aligned}
       \end{eqnarray}
    	where $\{x\}=x-\lfloor x\rfloor$ denotes the fractional part. Combining (\ref{4.9-4}), (\ref{4.95}) and Lemma \ref{local} with $a=0,b=\frac{n}{2}+1+\frac{1-\{\alpha+2s\}}{2},A=16$, yields
    	\begin{eqnarray}\label{4.9-6}
    	\begin{aligned}
    		|R_{i}(x_{\xi},t_{\xi})|\leq& C\eta^{i\gamma}
    		\sum_{j=1}^{i-1}\eta^{-j(k+\alpha)}\int_{-\infty}^{t_{\xi}}\int_{\mathbb{R}^{n}}|f-P|_{Q_{\eta^{j-1}}\backslash Q_{\eta^{j}}}(y,\tau)\frac{e^{-\frac{|x_{\xi}-y|^{2}}{16|t_{\xi}-\tau|}}}{|t_{\xi}-\tau|^{\frac{n}{2}+1+\frac{1-\{\alpha+2s\}}{2}}}dyd\tau\\
            \leq& C\eta^{i\gamma}\sum_{j=1}^{i-1}\eta^{j(-1+\{\alpha+2s\})}\left(\eta^{j(-n-2)}\int_{Q_{\eta^{j-1}}\backslash Q_{\eta^{j}}}|f-P|(y,\tau)dyd\tau\right)\\
    		\leq& C\eta^{i\gamma}\sum_{j=1}^{i-1}\eta^{-j(k+\alpha+1-\{\alpha+2s\})}\nu_{f}(\eta^{j-1})
            \leq C\eta^{i\gamma}\sum_{j=0}^{i-1}\eta^{(j+1)(-1+\{\alpha+2s\})}\left(\eta^{-j(k+\alpha)}\nu_{f}(\eta^{j})\right)
            \end{aligned}.
    \end{eqnarray}
    Hence, by the Abel transformation, we get
            \begin{eqnarray}\label{4.9-100}
    	\begin{aligned}
            |R_{i}(x_{\xi},t_{\xi})|\leq &C\eta^{i\gamma}\Bigg(\eta^{i(-1+\{\alpha+2s\})}\sum_{j=0}^{i-1}\eta^{-j(k+\alpha)}\nu_{f}(\eta^{j})-\\&\sum_{l=0}^{i-2}\left(\eta^{(l+1)(-1+\{\alpha+2s\})}-\eta^{l(-1+\{\alpha+2s\})}\right)\sum_{j=0}^{l}\eta^{-j(k+\alpha)}\nu_{f}(\eta^{j})\Bigg)\\
            \leq& C\eta^{i\gamma}\|f\|_{C_{1}^{k+\alpha}(0,0;1)}\left(i\eta^{i(-1+\{\alpha+2s\})}-\sum_{j=0}^{i-1}j\left(\eta^{(j+1)(-1+\{\alpha+2s\})}-\eta^{j(-1+\{\alpha+2s\})}\right)\right)\\
            =& C\eta^{i(k+\alpha+2s))}\|f\|_{C_{1}^{k+\alpha}(0,0;1)}\sum_{j=0}^{i-1}\eta^{j(1-\{\alpha+2s\})}
    		\leq C\eta^{i(k+\alpha+2s)}\|f\|_{C_{1}^{k+\alpha}(0,0;1)}.
    	\end{aligned}
    \end{eqnarray}
    Therefore, combining \eqref{exp-3.22}, \eqref{4.9-4} and \eqref{4.9-100}, we obtain \eqref{4.9-1}.

        Finally, we claim that for $0\leq j\leq\gamma$,
    	$$\eta^{ji}|D^{j}P_{i+1}(0,0)-D^{j}P_{i}(0,0)|\leq C\|f\|_{C_{1}^{k+\alpha}(0,0;1)}\eta^{i(k+\alpha+2s)}.$$
    	Indeed, by Lemma \ref{time-lemma}, Lemma \ref{3.7} and Lemma \ref{prop2.3}, we have
        \begin{eqnarray}
            \begin{aligned}
    		&\left|\frac{\partial^{|\zeta|+\beta} T_{\eta^{i-1}}}{\partial x^{\zeta}\partial t^{\beta}}(0,0)-\frac{\partial^{|\zeta|+\beta} T_{\eta^{i}}}{\partial x^{\zeta}\partial t^{\beta}}(0,0)\right|\\
            &=\left|\int_{-\infty}^{0}\int_{\mathbb{R}^{n}}\left(f-P\right)_{Q_{\eta^{i-1}}\backslash Q_{\eta^{i}}}(y,\tau)\frac{\partial^{|\zeta|+\beta}K_{-s}}{\partial x^{\zeta}\partial t^{\beta}}(-y,-\tau)dyd\tau\right|\\
    		&\leq C\sum_{i=0}^{\beta}\sum_{j=0}^{[\frac{|\zeta|}{2}]}\int_{-\infty}^{0}\int_{\mathbb{R}^{n}}|f-P|_{Q_{\eta^{i-1}}\backslash Q_{\eta^{i}}}(y,\tau)\frac{|y|^{2i+|\zeta|-2j}e^{-\frac{|y|^{2}}{4|\tau|}}}{|\tau|^{\frac{n}{2}+1-s+|\zeta|+\beta+i-j}}dyd\tau.
    	\end{aligned}
        \end{eqnarray}
        Similar to the discussion of $I_{i,j,l,|\zeta|}$, we have
    	\begin{eqnarray}\label{4.9-9}
    	\begin{aligned}
    		&\eta^{ji}|D^{j}P_{i+1}(0,0)-D^{j}P_{i}(0,0)|\\
    		\leq& \eta^{i(\gamma+1)}
    		\sum_{|\zeta|+2\beta=j}\sum_{m=0}^{\beta}\sum_{l=0}^{[\frac{|\zeta|}{2}]}C\int_{-\infty}^{0}\int_{\mathbb{R}^{n}}|f-P|_{Q_{\eta^{i-1}}\backslash Q_{\eta^{i}}}(y,\tau)\frac{|y|^{2m+|\zeta|-2l}e^{-\frac{|y|^{2}}{4|\tau|}}}{|\tau|^{\frac{n}{2}+1-s+|\zeta|+\beta+m-l}}dyd\tau\\
    		\leq& C\|f\|_{C_{1}^{k+\alpha}(0,0;1)}\eta^{i(k+\alpha+2s)}.
    	\end{aligned}
        \end{eqnarray}
        Hence, for each $0\leq j\leq\gamma$, the sequence $\{D^j P_l(0,0)\}_{l=1}^\infty$ is Cauchy and converges to some constants $a_{j}$. 
        Thus, we easily get estimate (\ref{4.9-3}), completing the proof.
    \end{proof}
\subsection{Estimate for $T_{r}$ with $2s+\alpha\in\Z$}
    In this subsection, we first give the estimate for $T_{r}$ with $2s+\alpha\in\Z$ and subsequently establish the pointwise regularity for $w_{1}$.  
    \begin{lemma}\label{lemma 3.10}
        For $\alpha\in[0,1)$ with $2s+\alpha\in\Z$, let $f\in C_{1}^{k+\alpha}(0,0;1)$($k\in\N$) be positive, and $T_{r}$ be defined by \eqref{def-T}. Then for any $\eta\in[1/4,1/2]$, there exist a sequence of polynomials $P_{i}\in\mathcal{P}_{\gamma}$, where $\gamma = k + \alpha + 2s$, and a positive constant $C=C(n,k,\alpha,s)$ such that for every integer $i\geq 1$,
    	\begin{enumerate}
    		\item[(1)] if $\gamma$ is odd, then for any $(x,t)\in Q_{\eta^{i}}$,
    	\begin{align}
    		|T_{\eta^{i-1}}(x,t)-P_i(x,t)|\leq Ci\|f\|_{C_{1}^{k+\alpha}(0,0;1)}\eta^{i(k+\alpha+2s-1)}|x|.
    	\end{align}
    	In addition, there exist constants $a_{j} ~(j=0,1,\cdots,\gamma-1)$ such that for all $i\geq1$,
    	\begin{align}
    		\sum_{j=0}^{\gamma}\eta^{j(i-1)}|D^{j}P_{i}(0,0)-a_{j}|\leq C\|f\|_{C_{1}^{k+\alpha}(0,0;1)}\eta^{i(k+\alpha+2s)};
    	\end{align}
        \item[(2)] if $\gamma$ is even, then
        \begin{align}
        	\|T_{\eta^{i-1}}(x,t)-P_i(x,t)\|_{L^{\infty}(Q_{\eta^{i}})}\leq Ci\|f\|_{C_{1}^{k+\alpha}(0,0;1)}\eta^{i(k+\alpha+2s)}.
        \end{align}
        In addition, there exist constants $a_{j} ~(j=0,1,\cdots,\gamma-1)$ such that for all $i\geq1$,
        \begin{align}
        	\sum_{j=0}^{\gamma}\eta^{j(i-1)}|D^{j}P_{i}(0,0)-a_{j}|\leq C\|f\|_{C_{1}^{k+\alpha}(0,0;1)}\eta^{i(k+\alpha+2s)}.
        \end{align}
        \end{enumerate}
    Moreover, assume $f\in C_{1}^{k+\alpha,Dini}(0,0;1)$, then
    \begin{align}
    	\|T_{\eta^{i-1}}(x,t)-P_i(x,t)\|_{L^{\infty}(Q_{\eta^{i}})}\leq C\|f\|_{C_{1}^{k+\alpha}(0,0;1)}\eta^{i(k+\alpha+2s)}.
    \end{align}
    In addition, there exist constants $a_{j} ~(j=0,1,\cdots,\gamma)$ such that for all $i\geq1$,
    \begin{align}
    	\sum_{j=0}^{\gamma}\eta^{j(i-1)}|D^{j}P_{i}(0,0)-a_{j}|\leq C\|f\|_{C_{1}^{k+\alpha}(0,0;1)}\eta^{i(k+\alpha+2s)},
    \end{align}
    \end{lemma}
    \begin{proof}
        The proof is similar to that of Lemma \ref{lemma 3.8}; therefore, we only indicate the necessary modifications.
        
        First, for cases (1) and (2), we replace \eqref{3.23} with the following definitions:
        \begin{eqnarray}
            \begin{aligned}
              P_{i}(x,t)&:=\sum_{j=0}^{\gamma-1}\sum_{|\sigma|=j}\frac{1}{\sigma!}\left(\frac{\partial^{\sigma} T_{\eta^{i-1}}}{\partial (x,t)^{\sigma}}(0,0)\right)\cdot(x,t)^{\sigma},\\
              R_{i}(x,t) &:=\sum_{|\sigma|=\gamma}\left(\frac{1}{\sigma!}\frac{\partial^{\sigma} T_{\eta^{i-1}}}{\partial (x,t)^{\sigma}}(x,t)\right)\cdot(x,t)^{\sigma},\\ 
            \end{aligned}
        \end{eqnarray}
        In case (1), since $\gamma$ is odd, for any multi-index $\sigma$ with $|\sigma|=\gamma$ satisfying
    $\sum_{j=1}^{n}\sigma_{j}+2\sigma_{n+1}=\gamma,$
     there exists some $1\leq j\leq n$ such that $\sigma_{j}\neq 0$. Consequently, \eqref{4.9-4}  is modified to
        \begin{align}
    				|R_{i}(x_{\xi},t_{\xi})| &=\left|\sum_{|\sigma|=\gamma}\frac{1}{\sigma!}\left(\frac{\partial^{\sigma} T_{\eta^{i-1}}}{\partial (x,t)^{\sigma}}(x_{\xi},t_{\xi})\right)\cdot(x_{\xi},t_{\xi})^{\sigma}\right|  \notag\\
    				&\leq C\eta^{i(k+\alpha+2s-1)}|x|\sum_{|\zeta|+2\beta=\gamma}\sum_{l=0}^{\beta}\sum_{j=0}^{[\frac{|\zeta|}{2}]}I_{i,j,l,|\zeta|},
    			\end{align}
                where $I_{i,j,l,|\zeta|}$ is given by \eqref{I}.  	

       Next, for all cases, we refine \eqref{4.95} and \eqref{4.9-6} by the estimate 
        \begin{eqnarray}
    	\begin{aligned}
            I_{i,j,l,|\zeta|}\leq&  C\sum_{j=1}^{i-1}\eta^{-j(k+\alpha)}\int_{-\infty}^{t_{\xi}}\int_{\mathbb{R}^{n}}|f-P|_{Q_{\eta^{j-1}}\backslash Q_{\eta^{j}}}(y,\tau)\frac{(|x_{\xi}-y|^{2}+|t_{\xi}-\tau|)^{\frac{k+\alpha}{2}}e^{-\frac{|x_{\xi}-y|^{2}}{8|t_{\xi}-\tau|}}}{|t_{\xi}-\tau|^{\frac{n}{2}+1-s+\frac{|\zeta|}{2}+\beta}}dyd\tau\\
    		\leq &C\sum_{j=1}^{i-1}\eta^{-j(k+\alpha)}\int_{-\infty}^{t_{\xi}}\int_{\mathbb{R}^{n}}|f-P|_{Q_{\eta^{j-1}}\backslash Q_{\eta^{j}}}(y,\tau)\frac{e^{-\frac{|x_{\xi}-y|^{2}}{16|t_{\xi}-\tau|}}}{|t_{\xi}-\tau|^{\frac{n}{2}+1}}dyd\tau\\
            \leq& C\sum_{j=1}^{i-1}\eta^{-j(k+\alpha+n+2)}\int_{Q_{\eta^{j-1}}\backslash Q_{\eta^{j}}}|f-P|(y,\tau)dyd\tau\\
    		\leq& C\sum_{j=1}^{i-1}\eta^{-(j-1)(k+\alpha)}\nu_{f}(\eta^{j-1})
            \leq C\sum_{j=0}^{i-1}\eta^{-j(k+\alpha)}\nu_{f}(\eta^{j}),
    	\end{aligned}
    \end{eqnarray}
    which leads to
    \begin{eqnarray}
        \begin{aligned}
            I_{i,j,l,|\zeta|}\leq
            \begin{cases}
                Ci\|f\|_{C_{1}^{k+\alpha}(0,0;1)}&~~\text{ when } f\in C_{1}^{k+\alpha}(0,0;1);\\
                C\|f\|_{C_{1}^{k+\alpha,Dini}(0,0;1)}&~~\text{ when } f\in C_{1}^{k+\alpha,Dini}(0,0;1).
            \end{cases}
        \end{aligned}
    \end{eqnarray}
        Thus, we complete the proof.
    \end{proof}
    Finally, combining Lemma \ref{lemma 3.8} and Lemma \ref{lemma 3.10}, we obtain Proposition \ref{w(x,t)}.
    \begin{proof}[Proof of Proposition \ref{w(x,t)}]
        We prove only case (a); the remaining cases follow by analogous arguments. The notation is the same as in the proof of Lemma \ref{lemma 3.8}.

         For a fixed $\eta\in[1/4,1/2]$, we define 
         \begin{eqnarray}\label{w-10}
         \begin{aligned}
             \widetilde{P}(x,t)=\sum\limits_{j=0}^{\gamma}\sum\limits_{|\sigma|=j}\frac{1}{\sigma!}a_{\sigma}(x,t)^{\sigma},
         \end{aligned}
         \end{eqnarray} where $a_{\sigma}$ is the component of $a_{|\sigma|}$.
        
        For any $(x,t)\in Q_{1/2}$, choose $i\in\N$ such that
    	$$\eta^{i+1}<|(x,t)|\leq\eta^{i}.$$
    	We denote $r=\eta^{i}$ and write
        $$w_{1}=S_{r}+T_{r}+u_{P},$$
        where $S_{r},T_{r}$ and $u_{P}$ are defined in \eqref{def-S}, \eqref{def-T} and \eqref{def-u}, respectively.
        
        Taking $P=\widetilde{P}+\bar{P}$, we get
    	\begin{eqnarray}
    	\begin{aligned}\label{w-2}
    	|w_{1}(x,t)-P(x,t)|\leq|T_{r}(x,t)-\widetilde{P}(x,t)|+|S_{r}(x,t)|+|u_{P}(x,t)-\bar{P}(x,t)|.
    	\end{aligned}
        \end{eqnarray}
    	First, from Lemma \ref{lemma 3.8}, we have
    	\begin{eqnarray}
    	\begin{aligned}\label{w-4}
    		|T_{r}(x,t)-\widetilde{P}(x,t)|&\leq |T_{r}(x,t)-P_{i+1}(x,t)|+|P_{i+1}(x,t)-\widetilde{P}(x,t)|\\
    		&\leq|T_{r}(x,t)-P_{i+1}(x,t)|+\sum_{j=0}^{\gamma}r^{j}|D^{j}P_{i+1}(0,0)-a_{j}|\\
    		&\leq C\|f\|_{C_{1}^{k+\alpha}(0,0;1)}r^{k+\alpha+2s}\leq C\|f\|_{C_{1}^{k+\alpha}(0,0;1)}|(x,t)|^{k+\alpha+2s}.
    	\end{aligned}
         \end{eqnarray}
        Next, Lemma \ref{ref-w()} gives
    	\begin{eqnarray}\label{w-3}
    	\begin{aligned}
    		\bbint_{Q_{r}}|S_{r}(y,\tau)|dyd\tau\leq C\|f\|_{C_{1}^{k+\alpha}(0,0;1)}r^{k+\alpha+2s}
    		\leq C\|f\|_{C_{1}^{k+\alpha}(0,0;1)}|(x,t)|^{k+\alpha+2s}.
    	\end{aligned}
        \end{eqnarray}
        Moreover, by Theorem \ref{3.10}, we obtain $u_{P}\in C^{k+\alpha+2s}(0,0;1/2).$ Consequently, there exist a polynomial $\bar{P}$ and a positive constant $C$ such that
        \begin{eqnarray}\label{w-1}
    	\begin{aligned}\|u_{P}-\bar{P}\|_{L^{\infty}(Q_{r})}+\|\bar{P}\|_{(0,0)}r^{k+\alpha+2s}\leq C\|f\|_{C_{1}^{k+\alpha}(0,0;1)}r^{k+\alpha+2s}.
        \end{aligned}
        \end{eqnarray}
        Combining (\ref{w-2}-\ref{w-1}), we obtain for any $r\in(0,1/2]$,
        \begin{eqnarray}\label{w-7}
        \begin{aligned}
            \bbint_{Q_{r}}|w_{1}(y,\tau)-P(y,\tau)|dyd\tau\leq C\|f\|_{C_{1}^{k+\alpha}(0,0;1)}r^{k+\alpha+2s}.
        \end{aligned}
        \end{eqnarray}
        
        We recall \eqref{w-10} and note that taking $i=1$ and $P_{1}=0$ in  \eqref{4.9-3} yields
    	\begin{eqnarray}\label{w-5}
        \begin{aligned}
\|\widetilde{P}\|_{(0,0)}\leq\sum_{j=0}^{\gamma}|a_{j}|\leq C\|f\|_{C_{1}^{k+\alpha}(0,0;1)}.
        \end{aligned}
        \end{eqnarray}
        Together with \eqref{w-1} and \eqref{w-5}, we obtain
        \begin{eqnarray}\label{w-8}
        \begin{aligned}
            \|P\|_{(0,0)}\leq \|\widetilde{P}\|_{(0,0)}+\|\bar{P}\|_{(0,0)}\leq C\|f\|_{C_{1}^{k+\alpha}(0,0;1)}.
        \end{aligned}
        \end{eqnarray}
         Thus, \eqref{w-7} and \eqref{w-8} imply \eqref{w-6}.
    \end{proof}
    Finally, we prove our main result.
    \begin{proof}[Proof of Theorem \ref{th6}]
       Without loss of generality, we assume $(x_{0},t_{0})=(0,0)$. We provide the proof for case (a) only, as the other cases can be handled similarly.
       
       Since $u=v_{1}+w_{1}$, where $v_{1}$ is defined by \eqref{est-v} and $w_{1}$ by \eqref{1.25}, we can apply Proposition \ref{v-E} to obtain that $v_{1}\in C^{\infty}(0,0;\frac{1}{2})$, and
       \begin{equation}\label{th6-v1}
           \|v_{1}\|_{C^{k+\alpha+2s}(0,0;\frac{1}{2})}\leq C\|u\|_{L^{\infty}(\widetilde{Q}_{\frac{1+\sqrt{n}}{2\sqrt{n}}})}.
       \end{equation}
       Furthermore, by Proposition \ref{w(x,t)}, we have $w_{1}\in C^{k+\alpha+2s}(0,0;\frac{1}{2})$, and
       \begin{equation}\label{th6-w1}
           \|w_{1}\|_{C_{1}^{k+\alpha+2s}(0,0;\frac{1}{2})}\leq C\|f\|_{C_{1}^{k+\alpha}(0,0;1)}.
       \end{equation}
       Combining \eqref{th6-v1} and \eqref{th6-w1} yields $u\in C_{1}^{k+\alpha+2s}(0,0;\frac{1}{2})$ and leads to the estimate \eqref{th6-u}.
    \end{proof}
    
    \section{Appendix: Equivalent Characterizations of Pointwise Function Spaces}\label{S4}
   In this section, we establish some equivalent characterizations of the pointwise function spaces. 
   We first give the usual definition of pointwise function spaces (\cite{L. Wang-1,L. Wang-2}).
    \begin{definition}[Pointwise function spaces]
    Let $U\subset\mathbb{R}^{n+1}$ be a bounded domain, let $(x_{0},t_{0})\in U$, $r_{0}>0$ with $Q_{r_{0}}(x_{0},t_{0})\subset U$, and let $f:U\to\mathbb{R}$ be a function. 
    For some $1\leq s_{1},s_{2}<\infty, k\in \N$ and weight function $\mu=\mu_{x_{0},t_{0}}(x,t):\R^{n}\times\R\to[0,\infty)$, we say that $f\in C_{s_{1},s_{2}}^{k,\mu(x,t)}(x_{0},t_{0};r_{0})$, if there exist a constant $K=K_{x_{0},t_{0}}>0$ and a polynomial $P=P_{x_{0},t_{0}}\in\mathcal{P}_{k}$ such that for all $(x_{1},t_{1})\in Q_{r_{0}}(x_{0},t_{0})$, it holds
    \begin{eqnarray}\label{2.1000}
    	\begin{aligned}
    		\left(\bbint_{t_{0}-r^{2}}^{t_{0}}
            \left(\bbint_{ B_{r}(x_{0})}|f(x,t)-P(x,t)|^{s_{1}}dx\right)^{\frac{s_{2}}{s_{1}}}dt\right)^{\frac{1}{s_{2}}}\leq K\mu(x_{1}-x_{0},t_{1}-t_{0}),
    	\end{aligned}
    \end{eqnarray}
    where $r=(|x_{1}-x_{0}|^{2}+|t_{1}-t_{0}|)^{\frac{1}{2}}$. In addition, if $s_{1}=\infty$, we replace the average integral with respect to $x$ in (\ref{2.1000}) into $\mathop{\esssup}_{B_{r}(x_{0})}$. Similarly, if $s_{2}=\infty$, we replace the average integral with respect to $t$ of (\ref{2.1000}) into $\mathop{\esssup}_{(t_{0}-r^{2},t_{0}]}$.

    The associated norm is defined as
    \begin{align}
    	\|f\|_{C_{s_{1},s_{2}}^{k,\mu(x,t)}(x_{0},t_{0};r_{0})}=\min_{K\geq0, P\in\mathcal{P}_{k}} (K+\|P\|_{x_{0},t_{0}}).
    \end{align}
    Note that if for $P=0$, there exists $K$ such that \eqref{2.1000} holds for all $0<r\leq r_{0}$, then we have $f\in C_{s_{1},s_{2}}^{k,\mu(x,t)}(x_{0},t_{0};r_{0})$ with polynomial $P=0$.

    In particular, we distinguish three important cases:
    \begin{enumerate}
    	\item[(1)] for $k\in\N, 0\leq \alpha<1$ and
    	$\mu_{1}(x,t)=(|x|^{2}+|t|)^{\frac{k+\alpha}{2}}$,
    we define $$ C_{s_{1},s_{2}}^{k,\alpha}(x_{0},t_{0};r_0)=C_{s_{1},s_{2}}^{k,\mu_{1}(x,t)}(x_{0},t_{0};r_{0});$$
    \item[(2)] for $k\in\N^{*}$ and
   $
    	\mu_{2}(x,t)=(|x|^{2}+|t|)^{\frac{k}{2}}|\ln (|x|^{2}+|t|)|,
    $
    we define $$ C_{s_{1},s_{2}}^{k,\ln}(x_{0},t_{0};r_{0})=C_{s_{1},s_{2}}^{k,\mu_{2}(x,t)}(x_{0},t_{0};r_{0});$$
    
    \item[(3)] for $k\in\N, 0\leq \alpha<1$ and 
    \begin{equation}\label{mu-3}
        \mu_{3}(x,t)=\nu_{f}(R):=\sup_{0<r\leq R}\left(\bbint_{t_{0}-r^{2}}^{t_{0}}\left(\bbint_{ B_{r}(x_{0})}|f(x,t)-P(x,t)|^{s_{1}}dx\right)^{\frac{s_{2}}{s_{1}}}dt\right)^{\frac{1}{s_{2}}},
    \end{equation}
    with $R=(|x|^{2}+|t|)^{\frac{1}{2}},$ we say that $f\in C_{s_{1},s_{2}}^{k,\alpha,Dini}(x_{0},t_{0};r_{0})$ if
    	\begin{align}
    		\int_{0}^{r_{0}}\frac{\nu_{f}(r)}{r^{k+\alpha+1}}dr+\nu_{f}(r_{0})<\infty.
    	\end{align}
    	In this case, we use another associated norm
    	\begin{align}
    		\|f\|_{C_{s_{1},s_{2}}^{k,\alpha,Dini}(x_{0},t_{0};r_{0})}=\nu_{f}(r_{0})+\int_{0}^{r_{0}}\frac{\nu_{f}(r)}{r^{k+\alpha+1}}dr+\|P\|_{(x_{0},t_{0})}.
    	\end{align}
    \end{enumerate}
    \end{definition}
    For simplicity, throughout this section we denote $C^{k,\alpha}_{\infty,\infty}(x_{0},t_{0})$ by $C^{k,\alpha}(x_{0},t_{0}),$  $C^{k,\alpha}_{s_{1},s_{1}}(x_{0},t_{0})$ by $C^{k,\alpha}_{s_{1}}(x_{0},t_{0})$ and $C^{k,\alpha}_{s_{1},s_{2}}(x_{0},t_{0})$ by $C^{k+\alpha}_{s_{1},s_{2}}(x_{0},t_{0})$ with analogous conventions for the other pointwise function spaces. We consider the case $s_{1}=s_{2}=1,r_{0}=1$. Now we are ready to give some equivalent characterizations for the above pointwise function spaces. 
    
    We first establish the equivalent characterizations for $C_{1}^{k+\alpha}(x_{0},t_{0};1)$.
    \begin{theorem}
    \label{reduce 1}
    	Let $f$ be a function and $P$ be a polynomial. For some $k\in\N$, $\alpha\in[0,1)$ and $(x_{0},t_{0})\in\R^{n}\times\R$, we define
    	\begin{equation}\label{C^k-1}
    		g(x,t)=\frac{f(x,t)-P(x,t)}{(|x-x_{0}|^{2}+|t-t_{0}|)^{\frac{k+\alpha}{2}}}
    	\end{equation}
        and let $\nu_{f}$ be given by \eqref{mu-3}.
    	Then the following statements are equivalent:
    	\begin{enumerate}
    		\item[(1)]$f\in C_{1}^{k+\alpha}(x_{0},t_{0};1)$ with polynomial $P;$
    		\item[(2)]$g\in C_{1}^{0}(x_{0},t_{0};1)$ with polynomial $0;$
            \item[(3)]there exists a positive constant $C=C(k,\alpha,f)$ such that for every $m\in\N^{*}$ and $r\in[1/4,1/2]$, there holds
            \begin{align}
                \sum_{i=0}^{m-1}r^{-i(k+\alpha)}\nu_{f}\left(r^{i}\right)\leq Cm;
            \end{align}
    	\end{enumerate}
    	Moreover, there exists a positive constant $C=C(n,k,\alpha)$ such that
    	\begin{align}
    		C^{-1}\|f\|_{C_{1}^{k+\alpha}(x_{0},t_{0};1)}\leq \|g\|_{C_{1}^{0}(x_{0},t_{0};1)}+\|P\|_{(x_{0},t_{0})}\leq C\|f\|_{C_{1}^{k+\alpha}(x_{0},t_{0};1)}.
    	\end{align}
    \end{theorem}
    \begin{proof}
    	Without loss of generality, we assume $(x_{0},t_{0})=(0,0)$ and $P=0$. 
        Note that
        $r^{-i(k+\alpha)}v_{f}(r^{i})\leq C\Longleftrightarrow f\in C_{1}^{k+\alpha}(x_{0},t_{0};1)$ and $\nu_{f}(r)$ is increasing with respect to $r$, we can easily obtain the equivalence between statements (2) and (3). 
        Consequently, it suffices to prove the equivalence between statements (1) and (2).

        $(2)\Longrightarrow (1)$: Since $g\in C_{1}^{0}(0,0;1)$,
    	for any $0<r\leq 1$, it holds,
    	\begin{eqnarray}
    		\begin{aligned}
    			\bbint_{Q_{r}}|f(x,t)|dxdt
    			&\leq \bbint_{Q_{r}}\frac{2^{\frac{k+\alpha}{2}}r^{k+\alpha}}{(|x|^{2}+|t|)^{\frac{k+\alpha}{2}}}|f(x,t)|dxdt=2^{\frac{k+\alpha}{2}}r^{k+\alpha}\bbint_{Q_{r}}|g(x,t)|dxdt\\
    			&\leq 2^{\frac{k+1}{2}}\|g\|_{C_{1}^{0}(0,0;1)}r^{k+\alpha},
    		\end{aligned}
    	\end{eqnarray}
    	which implies $f\in C_{1}^{k+\alpha}(0,0;1)$ and
    	$
    		\|f\|_{C_{1}^{k+\alpha}(0,0;1)}\leq 2^{\frac{k+1}{2}}\|g\|_{C_{1}^{0}(0,0;1)}.
    	$
  
        $(1)\Longrightarrow (2)$: By $f\in C_{1}^{k+\alpha}(0,0;1)$, for any $0<r\leq 1$, it holds
    	\begin{eqnarray}
    		\begin{aligned}\label{4.2-1}
    			&\int_{Q_{r}\backslash Q_{\frac{r}{2}}}|g(x,t)|dxdt
    			\leq
    			\frac{2^{k+\alpha}}{r^{k+\alpha}}\int_{Q_{r}\backslash Q_{\frac{r}{2}}}|f(x,t)|dxdt
    			\leq 2^{k+1}\|f\|_{C_{1}^{k+\alpha}(0,0;1)}r^{n+2}.
    		\end{aligned}  
    	\end{eqnarray}
    	Thus, we have
    	\begin{eqnarray}
    		\begin{aligned}
    			\int_{Q_{r}}|g(x,t)|dxdt
    			&=\sum_{i=0}^{\infty}\int_{Q_{\frac{r}{2^{i}}}\backslash Q_{\frac{r}{2^{i+1}}}}|g(x,t)|dxdt\\
    			&\leq 2^{k+1}\|f\|_{C_{1}^{k+\alpha}(0,0;1)}r^{n+2}\sum_{i=0}^{\infty}2^{-i(n+2)}\leq 2^{k+2}\|f\|_{C_{1}^{k+\alpha}(0,0;1)}r^{n+2},
    		\end{aligned}   
    	\end{eqnarray}
    	which yields $g\in C_{1}^{0}(0,0;1)$ and
    	$
    		\|g\|_{C_{1}^{0}(0,0;1)}\leq C\|f\|_{C_{1}^{k+\alpha}(0,0;1)}.
    	$  
        Hence, we complete the proof.
    \end{proof}
    Similarly, we obtain the equivalent result for $C^{k,\ln}(x_{0},t_{0})$.
    \begin{theorem}
    \label{reduce 3}
    	Let $f$ be a function and $P$ be a polynomial. For some $k\in\N$ and $(x_{0},t_{0})\in\R^{n}\times\R$, we define
    	$\nu_{f}$ and $g$ as \eqref{mu-3}, \eqref{C^k-1}, respectively.
    	Then the following statements are equivalent:
    	\begin{enumerate}
    		\item[(1)]$f\in C_{1}^{k,\ln}(x_{0},t_{0};1/2)$ with polynomial $P;$
    		\item[(2)]$g\in C_{1}^{0,\ln}(x_{0},t_{0};1/2)$ with polynomial $0;$
            \item[(3)]there exists a positive constant $C=C(k,\alpha,f)$ such that for every $m\in\N^{*}$ and any $r\in[1/4,1/2]$, there holds
            \begin{align}
                \sum_{i=1}^{m}r^{-(k+\alpha)i}\nu_{f}\left(r^{i}\right)\leq Cm^{2}.
            \end{align}
    	\end{enumerate}
    	Moreover, there exists a positive constant $C=C(n,k,\alpha)$ such that
    	\begin{align}
    		C^{-1}\|f\|_{C_{1}^{k,\ln}(x_{0},t_{0};1/2)}\leq \|g\|_{C_{1}^{0,\ln}(x_{0},t_{0};1/2)}+\|P\|_{(x_{0},t_{0})}\leq C\|f\|_{C_{1}^{k,\ln}(x_{0},t_{0};1/2)}.
    	\end{align}
    \end{theorem}
    Finally, we establish equivalent characterizations for $C_{1}^{k+\alpha,Dini}(x_{0},t_{0};1)$.
    \begin{theorem}
    \label{reduce 2}
    	Let $f$ be a function and $P$ be a polynomial. For some $k\in\N$, $\alpha\in[0,1)$ and $(x_{0},t_{0})\in\R^{n}\times\R$, we define
        \begin{equation}           
\nu_{g}(R):=\sup_{0<r\leq R}\bbint_{-r^{2}}^{0}\bbint_{ B_{r}}|g(x,t)|dxdt, 
\end{equation} 
    	and let $\nu_{f}$ and $g$ be defined as \eqref{mu-3}, \eqref{C^k-1} respectively.
    	Then the following statements are equivalent:
    	\begin{enumerate}
    		\item[(1)]$f\in C_{1}^{k+\alpha,Dini}(x_{0},t_{0};1)$ with polynomial $P;$
            \item[(2)] There exists a posive constant $C=C(k,\alpha,f)$ such that for all $r\in[1/4,1/2],$ it holds
            \begin{equation}
                \sum\limits_{i=0}^{\infty}r^{-i(k+\alpha)}\nu_{f}\left(r^{i}\right)\leq C.
            \end{equation}
    	\end{enumerate}
    \end{theorem}
    \begin{proof}
    	Without loss of generality, we assume $(x_{0},t_{0})=(0,0)$ and $P=0$. 
        We first show the equivalence between statements (1) and (2). 

        In the case $k+\alpha\neq 0,$ note that for any $r\in[1/4,1/2]$ and integer $i\in\N$ there exists a positive constant $C=C(k,\alpha,f)$ such that
        \begin{equation}
        C^{-1}\int_{r^{i+1}}^{r^{i}}\frac{\nu_{f}(s)}{s^{k+\alpha+1}}ds\leq \frac{\nu_{f}(r^{i})}{r^{i(k+\alpha)}}\leq C\int_{r^{i}}^{r^{i-1}}\frac{\nu_{f}(s)}{s^{k+\alpha+1}}ds.
        \end{equation}
        Consequently, we obtain
        \begin{equation}
            C^{-1}\left(\int_{0}^{1}\frac{\nu_{f}(s)}{s^{k+\alpha+1}}ds+\nu_{f}(1)\right)\leq \sum_{i=0}^{\infty}\frac{\nu_{f}(r^{i})}{r^{i(k+\alpha)}}\leq C\left(\int_{0}^{1}\frac{\nu_{f}(s)}{s^{k+\alpha+1}}ds+\nu_{f}(1)\right).
        \end{equation}
        For the case $k+\alpha=0$, we obtain the equivalence between statements (1) and (2) by
    	\begin{eqnarray}
    		\begin{aligned}
    			\sum\limits_{i=0}^{\infty}\nu_{f}\left(r^{i}\right)\leq\int_{0}^{\infty}\nu_{f}\left(r^{\theta}\right)d\theta+\nu_{f}(1)\leq C\left(\int_{0}^{1}\frac{\nu_{f}(s)}{s}ds+\nu_{f}(1)\right) \leq C\sum\limits_{i=0}^{\infty}\nu_{f}\left(r^{i}\right),
    		\end{aligned}
    	\end{eqnarray}
        \begin{remark}
            If we further assume that $\nu_{f}(2r)\geq 2^{k+\alpha}\nu_{f}(r)$ holds for all $r\in(0,1/2]$, then the above statements are equivalent to 
        \begin{enumerate}
            \item[(3)]$g\in C_{1}^{0,Dini}(x_{0},t_{0};1)$ with polynomial $0$.
        \end{enumerate}
        Moreover, there exists a positive constant $C=C(n,k,\alpha)$ such that
    	\begin{align}
    		C^{-1}\|f\|_{C_{1}^{k+\alpha,Dini}(x_{0},t_{0};1)}\leq \|g\|_{C_{1}^{0,Dini}(x_{0},t_{0};1)}+\|P\|_{(x_{0},t_{0})}\leq C\|f\|_{C_{1}^{k+\alpha,Dini}(x_{0},t_{0};1)}.
    	\end{align}
        \end{remark}
    \end{proof}

    {\bf{Acknowledgments.}} The authors would like to thank Professor Congming Li and Professor Jianli Liu for their useful suggestions.
 Guo is partially supported by the National Natural Science Foundation of China (Grant No. 12501145), the Natural Science Foundation of Shanghai (No. 25ZR1402207),   the China Postdoctoral Science Foundation (No. 2025T180838 and No. 2025M773061), the Postdoctoral Fellowship Program of CPSF (No. GZC20252004), and the Institute of Modern Analysis-A Frontier Research Center of Shanghai.
Shen and Xie are partially supported by the National Natural Science Foundation of China (Grant No. W2531006, 12250710674 and 12031012) and the Institute of Modern Analysis-A Frontier Research Center of Shanghai. 
 \medskip

{\bf{Date availability statement:}} Data will be made available on reasonable request.
\medskip

{\bf{Conflict of interest statement:}} There is no conflict of interest.

\bigskip

Yahong Guo

School of Mathematical Sciences

Shanghai Jiao Tong University

Shanghai, 200240, P.R. China

yhguo@sjtu.edu.cn
\medskip

Qizhen Shen

School of Mathematical Sciences

Shanghai Jiao Tong University

Shanghai, 200240, P.R. China

Department of Mathematics

Shanghai University

Shanghai 200444, P.R. China

zhuimengtiantang@sjtu.edu.cn
\medskip

Jiongduo Xie

School of Mathematical Sciences

Shanghai Jiao Tong University

Shanghai, 200240, P.R. China

jiongduoxie@outlook.com
\end{document}